\documentclass[11pt,a4paper]{article}
\usepackage{amssymb,latexsym}
\usepackage{amsrefs}
\usepackage[latin1]{inputenc}
\usepackage{amsmath,amssymb}
\usepackage{MnSymbol,wasysym}
\usepackage{graphicx}
\usepackage{a4wide}
\usepackage{color}
\usepackage{soul}
\usepackage{mathrsfs}
\usepackage{enumerate}

\usepackage{amsmath,amsthm}
\usepackage{tikz}
\usetikzlibrary{arrows}
\usetikzlibrary{decorations.markings}
\usetikzlibrary{matrix}
\usetikzlibrary{patterns}
\usetikzlibrary{fadings}
\numberwithin{equation}{section}
\newtheorem{theorem}{Theorem}[section]

\newtheorem{proposition}[theorem]{Proposition}
\newtheorem{lemma}[theorem]{Lemma}

\newtheorem{corollary}[theorem]{Corollary}
\newtheorem{definition}[theorem]{Definition}
\newtheorem{example}[theorem]{Example}

\newtheorem{question}{Question}

\newcommand{\aA}{{\mathcal A}}

\newcommand{\rR}{{\mathcal R}}
\newcommand{\sS}{\mbox{Shi}}

\newcommand{\lp}[2]{{\mathcal{L}^{#1}_{#2}}}
\newcommand{\dyp}[2]{{\mathcal{D}^{#1}_{#2}}}
\newcommand{\mincos}[2]{{\mathcal{M}_+({\widetilde{#1}_{#2}})}}
\newcommand{\al}[2]{\mathcal{A}(#1_{#2})}
\newcommand{\tc}{\textcolor}

\newcommand{\bj}{${\sf Bj}_1$}
\newcommand{\bjj}{${\sf Bj}_2$}

\newcommand{\fs}{\footnotesize}

\newcommand{\mshi}[1]{${\sf Shi}^m (#1)$}
\newcommand{\mnshi}[2]{{\sf Shi}^{#1}(#2)}
\newcommand{\shi}[1]{${\sf Shi}(#1)$}

\newcommand{\mcatn}[2]{{\sf N}^{#1}(#2)}
\newcommand{\ip}[2]{\left\langle #1,#2 \right\rangle}
\newcommand{\alc}{{\mathcal A}}
\newcommand{\bracket}{abacus }

\definecolor{mycyan}{rgb}{0.4,0.8,1.0}
\definecolor{mycyan2}{rgb}{0.0,0.48,0.65}
\definecolor{Green}{rgb}{0.2,0.7,0.2}

\newcommand{\pa}{{\mathcal P}}

\newcommand{\mgk}[2]{ \mathscr{K}_{#2}(#1)}
\newcommand{\w}[1]{w_{\scalebox{0.6}{$[#1]$}}}

\newcommand{\mk}[1]{{\color{mycyan2}#1}}
\newcommand{\hide}[1]{}
\usepackage{cancel,soul}
\usepackage[normalem]{ulem}
\newcommand{\pr}[1]{{\emph{pr}_{#1}}}

\newcommand{\fl}[1]{\left\lfloor #1 \right\rfloor}
\newcommand*\circled[1]{\tikz[baseline=(char.base)]{
    \node[shape=circle,draw,inner sep=1pt] (char) {#1};}}

\title{Bijections between generalized Catalan families \\of types
  $A$ and $C$}

\author{Myrto Kallipoliti\thanks{
Fak. f\"ur Mathematik, Universit\"at Wien, Oskar-Morgenstern-Platz 1, 1090
Wien, Austria,
\texttt{myrto.kallipoliti@univie.ac.at}
    }
  \, and Eleni Tzanaki\thanks{
    Department of Mathematics \& Applied Mathematics,
     University of Crete,
    GR-700 13 Voutes, Heraklion, Greece,
     \texttt{etzanaki@uoc.gr}
    }
}

\begin{document}
\maketitle

%\keywords{Shi hyperplane arrengements, abacus diagram, affine permutations,
%  lattice paths}

\begin{abstract}
  Motivated by the relation $\mcatn{m}{C_n}=(mn+1)\mcatn{m}{A_{n-1}}$,
  holding for the $m$-generalized Catalan numbers of type $A$ and $C$,
  the connection between  dominant regions of the
  $m$-Shi arrangement of type $A_{n-1}$ and $C_n$ is investigated.
     More precisely, it is explicitly shown  how 
    $mn+1$ copies of each element of the set $\rR_+^m(A_{n-1})$ of
    dominant regions of the $m$-Shi arrangement of type $A_{n-1}$, biject onto 
    the set  $\rR_+^m(C_{n})$ of type $C_{n}$ such regions.
   This is achieved by exploiting two different viewpoints to express
  the representative alcove of each region: the Shi tableau and the abacus
  diagram. In the same line of thought, a bijection between $mn+1$
  copies of each $m$-Dyck path of height $n$ and the set of 
   $N-E$ lattice paths inside an  $n\times mn$ rectangle is provided.
\end{abstract}

\section{Introduction}
\label{intro}

The classical Catalan numbers, $\mbox{Cat}(n)=\frac{1}{n+1}{2n\choose n}$,
constitute one of the most ubiquitous number sequences in
enumerative combinatorics, also appearing in several other contexts
varying from  algebra and representation theory
\cite{fz-ysga-03,fr-gcccc-05,sh-nost-97,sh-acawg-97} to discrete geometry
\cite{ath-gcn-04,cfz-prga-02,po-pab-05}.
In combinatorics only, we refer to \cite[Exesrcise 6.19]{st-vol2} for a list of
66 families of objects enumerated by $\mbox{Cat}(n)$.
Such objects (also called \emph{Catalan objects})
relevant to the present paper are, the
set of dominant regions of the Shi arrangement $\sS(A_{n-1})$,
 triangulations of a convex $(n+3)$-gon, and Dyck paths of length $n$.
In \cite{ath-gcn-04}, Athanasiadis generalized Catalan numbers for every
crystallographic root system  $\Phi$ and positive integer $m$.
More precisely,  he defined the {\it $m$-generalized Catalan} number of type
$\Phi$ as
\begin{equation}
\label{eq:genCat}
\mcatn{m}{\Phi}=\prod_{i=1}^n \frac{e_i+mh+1}{e_i+1},
\end{equation}
where  $n$ is the rank, $h$ is the Coxeter number and $e_i$ are the exponents
of $\Phi$. In particular, he showed that $\mcatn{m}{\Phi}$ counts the number
of \emph{dominant regions} of the \emph{$m$-Shi arrangement} associated to
$\Phi$ (see Section~\ref{sec:prel} for the undefined notions).
We note that the classical Catalan numbers $\mbox{Cat}(n)$ are indeed a special 
case of \eqref{eq:genCat}, since they
occur when $\Phi=A_{n-1}$ and $m=1$.
More generally, if $\Phi=A_{n-1}$ or $C_n$, and $m$ is an arbitrary positive
integer, the expression in \eqref{eq:genCat} reduces respectively to
\[\mcatn{m}{A_{n-1}}=\tfrac{1}{mn+1}\tbinom{(m+1)n}{n} \ \ \ \ \ \ \ \
\mbox{and}
\ \ \ \ \ \ \ \ \ \
\mcatn{m}{C_{n}}=\tbinom{(m+1)n}{n}.\]
The numbers $\mcatn{m}{A_{n-1}}$, also known as \emph{Fuss-Catalan numbers},
count a wealth of combinatorial objects, most of which can be seen as
$m$-generalizations of \emph{type $A$} Catalan objects.
For the families discussed above i.e., dominant regions of Shi arrangemets, 
polygon triangulations and Dyck paths, such $m$-generalizations have been an 
object of research for more than a decade. In the same spirit,
 generalized $(m,\Phi)$-Catalan objects  have been discovered,
 for every finite root system $\Phi$ and integer $m$
 (see for instance \cite{arm-phd-05,ath-gcn-04,fr-gcccc-05}).
In most cases, each root system  $\Phi$ was
studied separately, before a unified structure was discovered.
Almost always, the starting point  was  the relation
\begin{align}
\label{catAC}
\mcatn{m}{A_{n-1}}(mn+1)=\mcatn{m}{C_{n}},
\end{align}
holding between $m$-Catalan numbers of type $A$ and $C$.
Occasionally, an  $(m,C_n)$-Catalan object
is a type $A$ one, of certain size and symmetry,
where  $mn+1$  copies of the $(m,A_{n-1})$-Catalan object reside.
Although most of the times the symmetry is rather natural to guess or 
understand, locating the $mn+1$ copies of the $(m,A_{n-1})$-Catalan object
in the corresponding $(m,C_n)$-type, can vary from easy to very complicated.

To give a motivating example, we describe a class of  $(m,\Phi)$-Catalan objects
where Relation \eqref{catAC} arises trivially.
Consider the set ${\mathscr D}_{n}^m$ of $(m+2)$-angulations of a convex
$(mn+2)$-gon $P$ i.e., dissections of $P$ by noncrossing diagonals into
polygons each having $m+2$ vertices.  Let also ${\mathscr{C}}_{2n}^m$ be the 
subset of ${\mathscr D}_{2n}^m$  consisting  of centrally  symmetric  
$(m+2)$-angulations  of a $(2mn+2)$-gon. The sets ${\mathscr D}^m_n$ and 
${\mathscr C}^m_{2n}$ are combinatorial realizations of the facets of the
generalized cluster complex $\Delta^m(\Phi)$ for $\Phi=A_{n-1}$ and $C_n$
respectively \cite{fr-gcccc-05}.
From their  description, Relation \eqref{catAC} is evident:
\begin{figure}[h!]
  \centering
\begin{tikzpicture}[scale=0.5]

\begin{scope}[xshift=-4cm]
\path (10:2 cm) coordinate (A0); \path (30:2 cm) coordinate  (A1);
\path (50:2 cm) coordinate  (A2); \path (70:2 cm) coordinate  (A3);
\path (90:2 cm) coordinate  (A4); \path (110:2 cm) coordinate  (A5);
\path (130:2 cm) coordinate  (A6); \path (150:2 cm) coordinate  (A7);
\path (170:2 cm) coordinate  (A8); \path (190:2 cm) coordinate  (A9);
\path (210:2 cm) coordinate  (A10); \path (230:2 cm) coordinate  (A11);
\path (250:2 cm) coordinate  (A12); \path (270:2 cm) coordinate  (A13);
\path (290:2 cm) coordinate  (A14); \path (310:2 cm) coordinate  (A15);
\path (330:2 cm) coordinate  (A16); \path (350:2 cm) coordinate  (A17);

\draw (A0)--(A1)--(A2)--(A3)--(A4)--(A5)--(A6)--(A7)--(A8)--(A9)--
(A10)--(A11)--(A12)--(A13)--(A14)--(A15)--(A16)--(A17)--(A0);

\node  at (A0){\tiny $\bullet$}; \node  at (A1){\tiny $\bullet$};
\node  at (A2){\tiny $\bullet$}; \node  at (A3){\tiny $\bullet$};
\node  at (A4){\tiny $\bullet$}; \node  at (A5){\tiny $\bullet$};
\node  at (A6){\tiny $\bullet$}; \node  at (A7){\tiny $\bullet$};
\node  at (A8){\tiny $\bullet$}; \node  at (A9){\tiny $\bullet$};
\node  at (A10){\tiny $\bullet$}; \node  at (A11){\tiny $\bullet$};
\node  at (A12){\tiny $\bullet$}; \node  at (A13){\tiny $\bullet$};
\node  at (A14){\tiny $\bullet$}; \node  at (A15){\tiny $\bullet$};
\node  at (A16){\tiny $\bullet$}; \node  at (A17){\tiny $\bullet$};

\draw[line width=1.5 pt] (A4)--(A13);\draw (A4)--(A15);\draw (A2)--(A15);
\draw (A2)--(A17);\draw (A6)--(A13);\draw (A6)--(A11);\draw (A8)--(A11);
\end{scope}

\begin{scope}[xshift=4cm]
\path (10:2 cm) coordinate (A0); \path (30:2 cm) coordinate  (A1);
\path (50:2 cm) coordinate  (A2); \path (70:2 cm) coordinate  (A3);
\path (90:2 cm) coordinate  (A4); \path (110:2 cm) coordinate  (A5);
\path (130:2 cm) coordinate  (A6); \path (150:2 cm) coordinate  (A7);
\path (170:2 cm) coordinate  (A8); \path (190:2 cm) coordinate  (A9);
\path (210:2 cm) coordinate  (A10); \path (230:2 cm) coordinate  (A11);
\path (250:2 cm) coordinate  (A12); \path (270:2 cm) coordinate  (A13);
\path (290:2 cm) coordinate  (A14); \path (310:2 cm) coordinate  (A15);
\path (330:2 cm) coordinate  (A16); \path (350:2 cm) coordinate  (A17);

\draw (A4)--(A5)--(A6)--(A7)--(A8)--(A9)--(A10)--(A11)--(A12)--(A13)--(A4);

\node  at (A4){\tiny $\bullet$}; \node  at (A5){\tiny $\bullet$};
\node  at (A6){\tiny $\bullet$}; \node  at (A7){\tiny $\bullet$};
\node  at (A8){\tiny $\bullet$}; \node  at (A9){\tiny $\bullet$};
\node  at (A10){\tiny $\bullet$}; \node  at (A11){\tiny $\bullet$};
\node  at (A12){\tiny $\bullet$}; \node  at (A13){\tiny $\bullet$};

\draw[line width=1.5 pt] (A4)--(A13);
\draw (A6)--(A13);\draw (A6)--(A11);\draw (A8)--(A11);
\end{scope}
\end{tikzpicture}
\end{figure}
we  identify each centrally symmetric dissection $D$ in ${\mathscr C_{2n}^m}$
with the pair  consisting of its diameter and one copy of the two 
$(m+2)$-angulations of the $(mn+2)$-gon into which the diameter divides the
initial $(2mn+2)$-gon. Since the diameters are $mn+1$, the cardinality of  
${\mathscr C}_{2n}^m$  is indeed $(mn+1)\mcatn{m}{A_{n-1}}=\mcatn{m}{C_{n}}$.
\medskip

The aim of this work is to reveal two new instances of the identity 
$\mcatn{m}{A_{n-1}}(mn+1)=\mcatn{m}{C_{n}}$, which  might lead to a better 
geometric understanding of the type $A$ and $C$ Shi arrangements. More 
precisely, we provide an \emph{explicit} bijection  between each of the 
following pair of sets:
\begin{enumerate}
  \item[(${\sf Bj}_1$)]
  the set containing   $mn+1$ copies of each dominant region of the $m$-Shi
  arrangement \mshi{A_{n-1}} and that of dominant regions in  \mshi{C_{n}}
  \item[(${\sf Bj}_2$)] the set containing   $mn+1$ copies of each $m$-Dyck 
  path of height $n$ and that of lattice paths from $(0,0)$ to $(n,mn)$
  in the grid $n\times{mn}$.
\end{enumerate}

We note that the second bijection is based on an idea in \cite{mo-lpc-79} while
the first, which is the main contribution of this work, relies on the two
different ways to view the representative alcove of a region in \mshi{\Phi}:
its \emph{Shi tableau} \cite{fkt-fgcc-13} and its \emph{ abacus diagram}
\cite{fi-vaz-10,ahj-rcosc-13}. Each of the bijections \bj,\bjj\, can stand on
its own and the sections presenting them (Section \ref{sec:dom reg} and
\ref{sec:latt} respectively) can be read independently. However, in the setting
of dominant regions in \mshi{A_n}, there exists previous work
\cite{fkt-fgcc-13} which reveals a connection between the two bijections.
Unifying previous and current results, we have the following commutative diagram:
\begin{center}
  \begin{tikzpicture}
  \matrix (m) [matrix of math nodes,row sep=3em,column sep=4em,minimum
  width=2em]
  {
    \substack{\mbox{\fs dominant }
    \\\mbox{\fs regions in }}\mbox{\mshi{A_{n-1}}}\times[mn+1] &
     \substack{\mbox{\fs dominant }
       \\\mbox{\fs regions in }}\mbox{\mshi{C_{n}}}\\
%%%%%%%%%%%%%%%%%%%%%%%%%%%%%%%%%%%%%%%%%%%%%%%%%%%%%%%%%%%%%%%%%%%%%
       \substack{
        \mbox{\fs $m$-Dyck paths}\\
        \mbox{\fs of height $n$}}
    \times [mn+1] &
     \substack{
     \mbox{\fs lattice paths from }\\
     \mbox{\fs  $(0,0)$ to $(n,mn)$}\\
     \mbox{ \fs in the grid $n\times{mn}$} }
    \\
    %%%%%%%%%%%%%%%%%%%%%%%%%%%%%%%%%%%%%%%%%%%%%%%%%%%%%%%%%%%%%%%%%%%%%%
    \substack{\mbox{\fs $(m+2)$-angulations}\\\mbox{\fs of an $(mn+2)$-gon
      }}\times[mn+1] &
      \substack{
        \mbox{\fs centrally symmetric}\\
        \mbox{\fs $(m+2)$-angulations}\\
        \mbox{\fs of an $(2mn+2)$-gon }}
      %%%%%%%%%%%%%%%%%%%%%%%%%%%%%%%%%%%%%%%%%%%%%%%%%%%%%%%%%
       \\};
%%%%%%%%%%%%%%%%%%%%%%%%%%%%%%%%%%%%%%%%%%%%%%%%%%%%%%%%%%%
  \path[-stealth]
  (m-1-1) edge node [left] {\fs $\mbox{FKT}_1$}
  (m-2-1) edge node [above] {\fs ${\sf Bj}_1$} (m-1-2)
  (m-2-1.east|-m-2-2) edge node [above] {\fs ${\sf Bj}_2$} (m-2-2)
  (m-1-2) edge node [right] {$\exists$} (m-2-2)
  (m-2-1) edge node [left] {\fs $\mbox{FKT}_2$} (m-3-1)
  (m-3-1) edge node [above] {\fs ${\sf Bj}$} (m-3-2)
  (m-2-2) edge node [right] {$\exists$} (m-3-2);
  \end{tikzpicture}
\end{center}
where $\mbox{FKT}_1$ and $\mbox{FKT}_2$ are bijections given in
\cite{fkt-fgcc-13}.
\medskip

This paper is structured as follows. We end up this section by recalling
basic facts on root systems and  presenting the two families
of Catalan objects we are dealing with.
In Section~\ref{sec:dom reg} we introduce all necessary material and  build our
first bijection \bj. More precisely, in Section \ref{sec:mShi}
we include background on Shi arrangements, discuss the notion of
\emph{$m$-minimal alcoves} and explain their connection to dominant
regions in \mshi{\Phi}.
Subsequently, in Sections~\ref{sec:viewp} and \ref{sec:1line}, we describe
the two different viewpoints in which we can encode dominant alcoves:  Shi
tableaux and abacus diagrams. Section~\ref{sec:shivsbr} serves to clarify the
relation between them in the type A case and to show how this adjusts to the
type $C$ case. Section~\ref{m-min-criteria} provides criteria for $m$-minimality
in terms of the abacus representation. Finally, in Section~\ref{sec:proj}, we
construct the bijection \bj\ and present an  explicit example. We conclude with
Section \ref{sec:latt}, where we prove \bjj.

%%%%%%%%%%%%%%%%%%%%%%%%%%%%%%%%%%%%%%%%%%%%%%%%%%%%%%%%%%%%%%%%%%%%%%%%%%

\subsection{Preliminaries}
\label{sec:prel}
\medskip
\paragraph{Root systems.}
%\label{sec:roots}
Let $V$ be an $n$-dimensional Euclidean space with inner product
$\ip{\cdot}{\cdot}$. For each  non-zero $\alpha\in V,$ the  \emph{reflection} 
$r_{\alpha}$, is the linear map which sends $\alpha$ to its negative and fixes 
pointwise  the hyperplane $H_{\alpha}$ orthogonal to $\alpha$. A finite 
\emph{root system} $\Phi$ is a finite  collection of non-zero vectors in $V$, 
called roots, which span $V$ and satisfy  $\Phi\cap 
\mathbb{R}\alpha=\{\alpha,-\alpha\}$  and $r_{\alpha}\Phi=\Phi$ for all 
$\alpha\in \Phi$.  If, in addition, 
$2\tfrac{\ip{\alpha}{\beta}}{\ip{\alpha}{\alpha}}\in\mathbb{Z}$
for all $\alpha,\beta\in\Phi$,  then the root system  $\Phi$ is called {\em 
crystallographic}. Any hyperplane in $V$ not orthogonal to any root in $\Phi$, 
partitions $\Phi$ into two sets; the set $\Phi^+$ of {\em positive} and
$\Phi^-$ of {\em negative} roots, respectively.
The set $\Pi$ of \emph{simple roots} is a subset of $\Phi^+$
that spans $V$ and has the additional property that
each  positive root can be expressed as a linear combination of simple
roots  with non-negative  coefficients.
The \emph{rank} of $\Phi$ is the dimension of the
space generated by $\Phi^+$.

Since in this paper we deal with root systems of type $A$ and $C$,
we briefly describe their standard  choice of
positive $\Phi^+$ and simple $\Pi$ roots.
In what follows we denote by $\varepsilon_1,\ldots,\varepsilon_{n+1}$
 the standard basis of $\mathbb{R}^{n+1}$. For more information on root systems
 we refer the reader to \cite{hu-rgcg-90}.

For  $\Phi=A_n$, we have
$\Phi^+=\{\alpha_{ij}:=\varepsilon_i-\varepsilon_{j+1}\mid{}1\leq{}i\leq{}j\leq{}n\}$
and $\Pi=\{\alpha_i:=\varepsilon_i-\varepsilon_{i+1}\mid{}1\leq i\leq n\}$
respectively.
Then, each positive root can be written in terms of  simple roots
as:
\begin{align}
\label{typeAroots}
\alpha_{ij}=\alpha_i+\alpha_{i+1}+\cdots+\alpha_j,\ \mbox{for}\ 1\leq i\leq
j\leq n.
\end{align}

For $\Phi=C_n$, we have $\Phi^+=\{2\varepsilon_i, 2\varepsilon_n,
\varepsilon_i\pm\varepsilon_{j+1}\,|\,1\leq{}i\leq{}j\leq{}n-1\}$ and
$\Pi=\{\alpha_{i}:=\varepsilon_i-\varepsilon_{i+1},\alpha_n:=2\varepsilon_n\,|\,1\leq
i\leq n-1\}$ respectively. If, for $1\leq{i}\leq{j}\leq{n-1}$, we set 
$\alpha_{ij}:=\varepsilon_i-\varepsilon_{j+1}$, 
$\alpha_{in}:=\varepsilon_i+\varepsilon_n$ and 
$\overline\alpha_{ij}:=\varepsilon_i+\varepsilon_{j+1}$,
each positive root can be written  in terms of simple roots, as follows:
\begin{equation}
\begin{split}
\alpha_{ij}&=\alpha_i+\cdots+\alpha_j,\ \mbox{for}\ 1\leq i\leq j\leq n,\
\mbox{and}\\
\overline\alpha_{ij}&=\alpha_i+\cdots+\alpha_{j-1}+2\,
(\alpha_{j}+\cdots+\alpha_{n-1})+\alpha_n,
%\mbox{for}\ 1\leq i\leq j\leq n-1.
\end{split}
\label{typeCroots}
\end{equation}

where $\alpha_{ii}:=\alpha_{i}$ and
the sum $\alpha_i+\cdots+\alpha_{j-1}$ is empty for $i=j$.
\medskip

\paragraph{Generalized Catalan objects.}
\label{sec:cox obj}
We call \emph{generalized Catalan objects},
families of combinatorial objects which are counted by generalized
Catalan numbers. In this paragraph we present those
which are relevant to this paper:
the family of dominant regions of the $m$-Shi arrangement of
type $A$ and $C$, that of $m$-Dyck paths of height $n$, and
that of $N-E$ lattice paths inside an $n\times mn$ rectangle.
\smallskip

The {\em $m$-Shi arrangement} \mshi{\Phi}
associated to a crystallographic root system $\Phi$
and positive integer $m$,
is the collection of hyperplanes
$\{H_{\alpha,k} \mid \alpha \in \Phi^+,\ -m< k \leq m \}$,
where $H_{\alpha,k} = \{ v \in V \mid\ip{v}{\alpha}=k\}$,
 $\alpha \in \Phi$ and $k\in\mathbb{Z}$. The {\it dominant chamber} of $V$ is 
 the intersection $\bigcap_{\alpha\in\Phi_{>0}}\{v\in V\mid\left\langle 
 v,\alpha\right\rangle\ge 0\}$.
Every region contained in the dominant chamber is called  {\it dominant}.
We denote by $\rR_+^m(\Phi)$ the set of dominant regions in \mshi{\Phi}
and we note that  $\rR_+^m(\Phi)$ coincides with the set of  dominant
regions in the $m$-Catalan arrangement $\mbox{Cat}^m(\Phi)$.
The number of regions in $\rR_+^m(\Phi)$ is $\mcatn{m}{\Phi}$.
We refer the reader to \cite{ath-gcn-04} for more details.

A \emph{$N-E$ lattice path}, is a lattice path in the grid $\mathbb{Z} \times 
\mathbb{Z}$
which takes only North $(0,1)$ and East $(1,0)$ steps.
An {\em $m$-Dyck path of height $n$}, or an {\em $(m,n)$-Dyck path} for short,
is a $N-E$ lattice path from $(0,0)$ to $(mn,n)$ that does not go below the line
$y=\frac{1}{m}x$. Finally, an \emph{$N-E$ lattice path inside an \emph{$n\times mn$}
rectangle}, is an $N-E$ lattice path from $(0,0)$ to $(mn,n)$.
The number of $(m,n)$-Dyck paths of height $n$ is $\mcatn{m}{A_{n-1}}$,
while that of $N-E$ lattice paths inside an $n\times mn$ rectangle is $\mcatn{m}{C_n}$.
Although  there is not a theory establishing type $(m,\Phi)$-lattice paths,
for our purposes we consider the above sets as instances
of type $A$ and $C$ Catalan objects, avoiding further generalizations.

\section{Dominant regions and \bj}
\label{sec:dom reg}

\subsection{Shi arrangement and dominant alcoves}
\label{sec:mShi}

For $m\to\infty$, the arrangement   \mshi{\Phi}
is the infinite collection of hyperplanes
$H_{\alpha,k}$, $\alpha\in\Phi^+,k\in\mathbb Z$.
The arrangement ${\sf Shi}^{\infty}(\Phi)$, which we simplify to \shi{\Phi},
is called \emph{ the Shi arrangement of $\Phi$}.
All regions in \shi{\Phi} are simplices, called {\em alcoves}.
The set of {\em dominant alcoves} i.e., those  lying in the dominant chamber,
is denoted by $\aA_+(\Phi)$.
The \emph{fundamental alcove} $\aA_0$ is the unique
dominant alcove whose closure $\bar{\aA}_0$ contains the origin.
For a fixed finite crystallographic root system $\Phi$, we
denote by $S_a$ the set of  reflections through all hyperplanes  in  \shi{\Phi}.
The \emph{affine Weyl group} $W_a$ of
$\Phi$ is the infinite Coxeter group  generated by the reflections in $S_a$.
The group $W_a$ acts simply transitively on the set of alcoves in
\shi{\Phi}, thus one can identify each alcove $\aA$ with the
unique $w\in W_a$ for which $\aA=w\aA_0$.
%% \et{why $w^{-1}$ and not $w$?}
The above bijection restricts to one between
the set $\mathcal{M}_+(W_{\alpha})$ of minimal length coset
representatives in $W_a/W$ and that of dominant alcoves in \shi{\Phi}.
In other words, $w\alc_0\in\alc_+(\Phi)$
if and only if $w\in$ $\mathcal{M}_+(W_{\alpha})$.

Clearly, each dominant region
$\rR\in\rR_+^m(\Phi)$ consists of one or more
alcoves. However, among them there exists a unique, denoted by $\alc_{\rR}$,
which is closest to the origin. We call this the \emph{$m$-minimal} alcove of
$\rR$. Every hyperplane $H_{\alpha,k}$ that supports a facet of $\rR$ is called
a \emph{wall}. We say that $H_{\alpha,k}$ is a \emph{separating wall} if $\rR$
and $\aA_0$ lie in different halfspaces delimited by $H_{\alpha,k}$.
It is proved in  \cite[Section 3]{ath-rgcn-05} that each region  $\rR$ in
\mshi{\Phi} and its minimal alcove $\alc_{\rR}$	have the same set of separating
walls. We refer the reader to \cite{ath-gcn-04, ath-
  rgcn-05, atz-pc-06}  for more details.

Our goal now is to exploit two different viewpoints
of the dominant alcoves in \shi{A_{n-1}} and \shi{C_n}: (i)
the \emph{Shi tableau} and (ii) the \emph{abacus diagram}. These combinatorial
structures, combined with the fact that each region in \mshi{\Phi}
is represented by its $m$-minimal alcove,
will be our tools for constructing \bj.
\subsection{Combinatorial viewpoint I (Shi tableaux for dominant alcoves)}
\label{sec:viewp}
Based on an idea of Shi, who arranges the positive roots of $\Phi^+$ in
diagrams \cite{sh-nost-97},  we assign to each dominant alcove $\alc$, a
set $\{k_\alpha,\alpha\in\Phi^+\}\subset\mathbb{Z}$ of
coordinates  which describe its location in the affine Shi arrangement
\shi{A_{n-1}}.
In particular, each $k_\alpha$ counts the number of positive integer 
translations of the hyperplane $H_{\alpha,0}$, which separate $\alc$ from the 
origin. The set of these coordinates, arranged according to the corresponding  
root diagram, is called the \emph{Shi tableau of the alcove} $\alc$.
We note that the entries of the Shi tableau yield the face defining
inequalities of the alcove. More precisely, for each
$\alpha\in\Phi^+$ we have
\begin{equation}
\label{equ:face_ineq_alcove}
k<\ip{x}{\alpha}<k+1 \mbox{ for all }x\in \aA
\;\;\;\mbox{ if and only if }\;\;\;k_{\alpha}=k.
\end{equation}
The above inequalities imply that the coordinates of the Shi tableau
have to satisfy  the so-called {\em Shi conditions} i.e., for every
$\alpha,\beta,\gamma
\in\Phi^+$ with $\alpha+\beta=\gamma$
\begin{equation}
\label{equ:general{}shi{}cond}
k_{\gamma}=k_{\alpha}+k_{\beta}+\delta,\,
\mbox{ where }\,\delta=\delta(\alpha,\beta,\gamma) \in\{0,1\}.
\end{equation}

The Shi tableau of a dominant alcove $\alc$ in \shi{A_n}
can be seen as a staircase Young diagram
of shape $(n,n-1,\dots,1)$ whose entries are non-negative integers
$k_{ij}$,\,$1\leq i\leq j\leq n$ (see Figure~\ref{fig:all_type_tableaux},
left). In particular, the coordinate  $k_{ij}$ occupying the cell  $(i,n-j+1)$
\footnote{ The cell $(i,j)$ lies in the $i$-th row and $j$-th column of the
  diagram, where rows are counted from top to bottom and columns from left to
  right.}
of the Shi tableau, indicates the number of integer translates of the
hyperplane $H_{\alpha_{ij},0}$ that  separate the alcove $\alc$ from the origin.
The Shi conditions in this case become:
\begin{equation}
k_{ij}=k_{i\ell}+k_{\ell+1\,j}+\delta_{i\ell}\,
\mbox{ with }\,\delta_{i\ell} \in\{0,1\},
\end{equation}
for every $1\leq{}i<j\leq{}n$ and all
$i\leq\ell<j$.

Traditionally, the Shi tableau of an alcove in \shi{C_n} is represented by a 
shifted Young diagram. Below, we describe a \emph{different but equivalent} way
to write the Shi coordinates. We do this to emphasize and exploit its
relation to the type $A$ case. We arrange the coordinates in a staircase diagram
of shape $(2n-1,2n-2,\dots,1)$ as follows: for $1\leq i\leq j\leq n-1$ the
coordinate $\overline k_{ij}$ occupies the cell $(i,j)$, whereas for
$1\leq{}i\leq{}j\leq{}n$ the coordinate $k_{ij}$ occupies the cell $(i,2n-j)$.
We finally fill the remaining cells so that the diagram is self-conjugate.
Notice that the tableau obtained this way is a self-conjugate Shi
tableau of type $A_{2n-1}$.
%%%%%%%%%%%%%%%%%%%%%%%%%%%%%%%%%%%%%%%%%%%%%%%%%%%%%%%%%%%%%%%%%%%%%%%%%%

%%%%%%%%%%%%%%%%%%%%%%%%figure with all type tableaux%%%%%%%%%%%%%%%%%%%%%%%%%%
\begin{figure}[h]
  \begin{center}
  \begin{tikzpicture}[scale=0.65]
  %%%%%%%%%%%%%%%%%%%%%%type A_4 tableau
  %%%%%%%%%%%%%%%%%%%%%%%%%%%%%%%%%%%%%%%%%%%%%%%%%%%%%%%%%%%%%%%%%%%%%%%%%%%%%%%%%%%%%%%%%%%%%%%%%%%%%%
  \begin{scope}[scale=0.85,xshift=0cm,yshift=0cm]
  \draw[line width= 0.5pt](0,0)rectangle(1,4);
  \draw[line width= 0.5pt](1,1)rectangle(2,4);
  \draw[line width=0.5pt](2,2)rectangle(3,4);
  \draw[line width= 0.5pt](3,3)rectangle(4,4);
  \draw[line width= 0.5pt](0,1)--(1,1);
  \draw[line width= 0.5pt] (0,2)--(2,2);
  \draw[line width= 0.5pt] (0,3)--(3,3);
  \node at(.5,3.5){$k_{14}$};\node at(1.5,3.5){$k_{13}$};\node at(2.5,3.5)
  {$k_{12}$};\node at(3.5, 3.5){$k_{11}$};
  \node at(.5,2.5){$k_{24}$};\node at(1.5,2.5){$k_{23}$};\node
  at(2.5,2.5){$k_{22}$};\node at(.5,1.5){$k_{34}$};\node
  at(1.5,1.5){$k_{33}$};\node at(.5,.5){$k_{44}$};
  \end{scope}
  %%%%%%%%%%%%%%%%%%%%%%type C tableau%%%%%%%%%%%%%%%%%%%%%%%%%%%%%%%%%%%%%%%%%%
  \begin{scope}[scale=0.85,xshift=15cm,yshift=0cm]
  \foreach\i in{-3,...,3}{\draw(\i,3)rectangle(1+\i,4);}
  \foreach\i in{-2,...,2}{\draw(\i,2)rectangle(1+\i,3);}
  \foreach\i in{-1,...,1}{\draw(\i,1)rectangle(1+\i,2);}
  \foreach\i in{0,...,0}{\draw(\i,0)rectangle(1+\i,1);}
  \node at(-2.5,3.5){$\overline k_{11}$};
  \node at(-1.5,3.5){$\overline k_{12}$};
  \node at(-0.5,3.5){$\overline k_{13}$};
  \node at(0.5,3.5){$k_{14}$};\node at(1.5,3.5){$k_{13}$};
  \node at (2.5,3.5) {$k_{12}$};\node at(3.5,3.5){$k_{11}$};
  \node at (-1.5,2.5){$\overline k_{22}$};
  \node at(-0.5,2.5){$\overline k_{23}$};
  \node at(0.5,2.5){$k_{24}$};\node at(1.5,2.5){$k_{23}$};
  \node at(2.5,2.5){$k_{22}$};
  \node at(-0.5,1.5){$\overline k_{33}$};
  \node at(0.5,1.5){$k_{34}$};\node at(1.5,1.5){$k_{33}$};
  \node at (0.5,0.5) {$k_{44}$};
  \foreach\i in{-3,...,3}{\draw[black!50!](-3,\i)rectangle(-2,\i+1);}
  \foreach\i in{-2,...,2}{\draw[black!50](-2,\i)rectangle(-1,1+\i);}
  \foreach\i in{-1,...,1}{\draw[black!50](-1,\i)rectangle(0,1+\i);}
  \foreach\i in{0,...,0}{\draw[black!50](0,\i)rectangle(1,1+\i);}
  \node at(-2.5,-2.5) {$\tc{black!50}{k_{11}}$};
  \node at(-2.5,-1.5){$\tc{black!50}{k_{12}}$};
  \node at(-2.5,-0.5){$\tc{black!50}{k_{13}}$};
  \node at(-2.5,0.5){$\tc{black!50}{k_{24}}$};
  \node at(-2.5,1.5){$\tc{black!50}{\overline k_{13}}$};
  \node at(-2.5,2.5){$\tc{black!50}{\overline k_{12}}$};
  \node at(-1.5,1.5){$\tc{black!50}{\overline k_{23}}$};
  \node at(-1.5,0.5){$\tc{black!50}{k_{14}}$};
  \node at(-1.5,-0.5){$\tc{black!50}{k_{23}}$};
  \node at(-1.5,-1.5){$\tc{black!50}{k_{22}}$};
  \node at(-0.5,0.5){$\tc{black!50}{k_{34}}$};
  \node at(-0.5,-0.5){$\tc{black!50}{k_{33}}$};
  \end{scope}

  \end{tikzpicture}
  \caption{The Shi tableau for type $A_4$(left) and $C_4$(right).
    Replacing each coordinate $k_{ij}$ by the corresponding root $\alpha_{ij}$,
    we obtain the root diagram. Notice that, in both diagrams, the simple roots 
    are those located on the diagonal, while  each positive root is the sum 
  of the simple roots lying below and to its right (this agrees with 
  expressions \eqref{typeAroots} and \eqref{typeCroots}).
        }
  \label{fig:all_type_tableaux}
\end{center}
\end{figure}

\begin{figure}
\begin{center}
\begin{tikzpicture}[scale=1.1]
%define the endpoints for the orthogonal hyperplanes

\path (0:6cm) coordinate (+O2); \path (0:-3cm) coordinate (-O2);  
\node at (0:6.5cm) {$ \tc{black}{H_{\alpha_2,0}}$};
\begin{scope}[yshift= 0.87 cm]
\path (0:6cm) coordinate (+T2); \path (0:-3cm) coordinate (-T2); \node at 
(0:6.5cm) {$ \tc{black}{H_{\alpha_2,1}}$};
\end{scope} 
\begin{scope}[yshift= 1.74 cm]
\path (0:6cm) coordinate (+TT2); \path (0:-3cm) coordinate (-TT2); \node at 
(0:6.5cm) {$ \tc{black}{H_{\alpha_2,2}}$};
\end{scope} 
\begin{scope}[yshift= 2.61 cm]
\path (0:6cm) coordinate (+TTT2); \path (0:-3cm) coordinate (-TTT2); \node at 
(0:6.5cm) {$ \tc{black}{H_{\alpha_2,3}}$};
\end{scope}

\path (120:4cm) coordinate (+O12); \path (120:-1.2cm) coordinate (-O12); \node 
at (120:4.2 cm) { \rotatebox{0}{ $ \tc{black}{H_{\alpha_{12},0}}$ }  };
\begin{scope}[xshift= 1cm] 
\path (120:4cm) coordinate (+T12); \path (120:-1.2cm) coordinate (-T12); \node 
at (120:4.2cm) {\rotatebox{0}{ $ \tc{black}{H_{\alpha_{12},1}} $ }  };
\end{scope} 
\begin{scope}[xshift= 2cm] 
\path (120:4cm) coordinate (+TT12); \path (120:-1.2cm) coordinate (-TT12); 
\node at (120:4.2 cm) {\rotatebox{0}{ $ \tc{black}{H_{\alpha_{12},2}} $ }  };
\end{scope} 
\begin{scope}[xshift= 3cm] 
\path (120:4cm) coordinate (+TTT12); \path (120:-1.2cm) coordinate (-TTT12); 
\node at (120:4.2cm) {\rotatebox{0}{ $ \tc{black}{H_{\alpha_{12},3}} $ }  };
\end{scope}

\path (60:5cm) coordinate (+O1); \path (60:-1.3cm) coordinate (-O1); \node at 
(60:5.2cm) { \rotatebox{0}{$ \tc{black}{ H_{\alpha_1,0}}$}};
\begin{scope}[xshift=1 cm]
\path (60:5cm) coordinate (+T1); \path (60:-1.3cm) coordinate (-T1);  \node at 
(60:5.2cm) { \rotatebox{0}{$ \tc{black}{H_{\alpha_1,1}}$}}; 
\end{scope} 
\begin{scope}[xshift=2 cm]
\path (60:5cm) coordinate (+TT1); \path (60:-1.3 cm) coordinate (-TT1);  \node 
at (60:5.2cm) { \rotatebox{0}{$ \tc{black}{H_{\alpha_1,2}}$}};  
\end{scope}
\begin{scope}[xshift=3 cm]
\path (60:5.2cm) coordinate (+TTT1); \path (60:-1.3 cm) coordinate (-TTT1);  
\node at (60:5.2cm) { \rotatebox{0}{$ \tc{black}{H_{\alpha_1,3}}$}}; 
\end{scope}

          % the orthogonal hyperplanes
      \draw[line width= 1pt, color=black] (-O2) -- (+O2);
       \draw[line width= 1pt, color=black] (-O12) -- (+O12);
         \draw[line width= 1pt, color=black] (-O1) -- (+O1);
  \draw[line width= 0.5pt, color=black] (-T12) -- (+T12);
 \draw[line width= 0.5pt, color=black] (-T2) -- (+T2);
  \draw[line width= 0.5pt, color=black] (-T1) -- (+T1);
 \draw[line width= 0.5pt, color=black] (-TT2) -- (+TT2);
 \draw[line width= 0.5pt, color=black] (-TT12) -- (+TT12);
  \draw[line width=0.5pt, color=black] (-TT1) -- (+TT1);
  \draw[line width=0.5pt, color=black] (-TTT2) -- (+TTT2);
 \draw[line width=0.5, color=black] (-TTT12) -- (+TTT12);
  \draw[line width=0.5, color=black] (-TTT1) -- (+TTT1);
  
  \filldraw[color=mycyan!10!gray,opacity=0.15] (+O1)--(34.4:7.6)--(+O2)--(0,0);

\fill[pattern=north west lines, pattern color=mycyan,xshift=
3.02cm,yshift=1.74cm] (60:1)--(0:1)--(0,0);

\node at (30:4.4) {$\mathcal R$};
\draw [mycyan,thick, -stealth] (27:8)  to[out=-100,in=80, 
distance=1.7cm ] (30:4.6);

 \begin{scope}[xshift=7.5 cm,yshift=3cm]
 \node at (1.4,1.4){\fs Shi tableau of: };
 \node at (0.5,1){\fs the region };
 \draw[line width= 0.5pt](0,0)rectangle(0.4,0.4);
 \draw[line width= 0.5pt](0,0.4)rectangle(0.8,0.8);
 \draw[line width= 0.5pt](0.4,0.4)--(0.4,0.8);
 \node at (0.2,0.2){2};
 \node at (0.2,0.6){3};
 \node at (0.6,0.6){2};
 \end{scope}
 
 \begin{scope}[xshift=9.5 cm,yshift=3cm]
  \node at (0.8,1){\fs $\&$ its minimal alcove };
 \draw[line width= 0.5pt](0,0)rectangle(0.4,0.4);
 \draw[line width= 0.5pt](0,0.4)rectangle(0.8,0.8);
 \draw[line width= 0.5pt](0.4,0.4)--(0.4,0.8);
 \node at (0.2,0.2){2};
 \node at (0.2,0.6){4};
 \node at (0.6,0.6){2};
 \end{scope}
 
\end{tikzpicture}
\end{center}
	\caption{The arrangement $\mbox{Shi}^3(A_2)$. The dashed alcove is a
    $3$-minimal alcove since,  among the two alcoves contained in $\rR$,
    it is the one closest to the origin.
    The Shi tableaux of $\rR$ and $\aA_{\rR}$, are shown on the	right.}
  \label{fig:ex_tableau}
\end{figure}
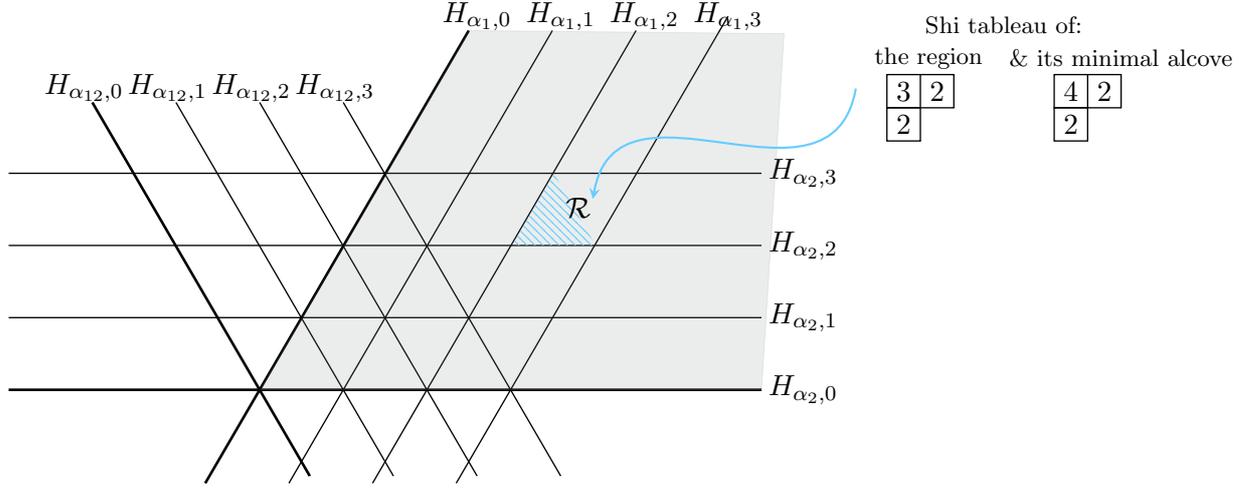

\subsection{Combinatorial viewpoint II (The abacus diagram)}
\label{sec:1line}
%\subsubsection*{Type $A$}
%\label{ss:typeA}
A standard way to realize the affine group
$\widetilde A_{n-1}$  is by identifying it with the set of all
$\mathbb{Z}$-permutations i.e., all bijections $w$ of $\mathbb{Z}$ in itself
such that:
\begin{itemize}
\item[(i)]
 $w(x+n)=w(x)+n$ for all $x\in \mathbb{Z}$, and
\item[(ii)]
 $\sum\limits_{i=1}^{n} w(i)=\binom{n+1}{2}$,
\end{itemize}
with composition as group operation.
Clearly, each such permutation is uniquely determined by its
values on $n$ consecutive integers.
Thus, for each $k\in\mathbb{Z}$ we write $w=[w(k+1),\ldots,w(k+n)]$
and call  this a \emph{window} of $w$. In the special case where $k=0$,
we say that  $[w(1),\dots,w(n)]$ is  the \emph{base window} of $w$.
In this paper we will often encode the  base window of a permutation
$w\in\widetilde A_{n-1}$ in a  more concise form, which is actually
equivalent to the so-called \emph{abacus representation} of $w$ 
\cite{be-bra-vaz-09,fi-vaz-10}.
To do this, for each $a\in\mathbb{Z}$
we set  $r=a\!\!\mod{n}$, $\ell=\left\lfloor\frac{a}{n}\right\rfloor$
and we write  $a:=r^\ell$ to imply that  $a=r+n\ell$.
%\begin{equation}
%\label{level}
% a:=r^\ell  \;\;\;\;\mbox{ meaning  that }\;\;\;
% a=r+\ell n\;\; \mbox{ where  } \;\; r=a\!\!\!\mod{n} \;\;\mbox{ and }\;\;
% \ell=\left\lfloor\frac{a}{n}\right\rfloor
%\end{equation}
We say that $r$ is the \emph{base} and $\ell$  the \emph{level} of $a$.
Note that, although $n$ is suppressed in the base-level notation, it is
implicit  from the rank of the affine group in which the permutation $w$ 
belongs. In
other words, when we write
$w=[r_1^{\ell_1},\ldots,r_n^{\ell_n}]\in\widetilde A_{n-1}$, we tacitly
understand that $r_i^{\ell_i}=r_i+n\ell_i$.
The \emph{abacus diagram} or simply \emph{abacus} of $w$
is the base window written in the concise form
$[r_1^{\ell_1},\dots,r_{n}^{\ell_{n}}]$
\footnote{This diagram corresponds to the flush abacus with $n$
runners, whose $i$-th runner has defining bead  on level $\ell_i$.
Abacus models encode nicely the action of the affine group $W_{\alpha}$
in the set of alcoves $\aA(\Phi)$, when $\Phi$ is a classical root system
\cite{han-bra-12}.
This justifies the term {\em abacus},
which we use with no further reference to its structure and properties.
}.
Then,  conditions {\rm(i)} and {\rm(ii)} imply that:
\begin{equation}
\label{equ:bracket}
\{r_1,\ldots,r_n\}=\{1,\dots,n\} \;\;\mbox{ and }\;\; \ell_1+\cdots+\ell_n=0.
\end{equation}
The \emph{level vector} $\vec{n}(w)=(\beta_1,\ldots,\beta_n)$ of
$w$ is the vector whose  $i$-th coordinate $\beta_i$ is equal to the level of
$i$ in the base window of $w$.
%the unique $r_j$ in the base window  for which $r_j=i$.
Clearly $\{\beta_1,\ldots,\beta_n\}=$ $\{\ell_1,\ldots,\ell_n\}$
and $\{1^{\beta_1},\ldots,n^{\beta_n}\}=\{r_1^{\ell_1},\ldots,r_n^{\ell_n}\}$.
 If, in addition, the entries of the base window are sorted i.e.,
$w(1)<\cdots<w(n)$, then $w$ belongs to the set $\mincos{A}{n-1}$
 of minimal length  coset representatives of $\widetilde{A}_{n-1}/A_{n-1}$.
 In fact, it turns out that all permutations  $w\in\widetilde A_{n-1}$
 whose base window is sorted,   biject to the set $\mincos{A}{n-1}$
(see \cite[Chapter 8.3]{bb_ccg_04}).
% Since the sets $\dal{A}{n-1}$ and $\mincos{A}{n-1}$
%are in bijection,  we write $w=[w(1),\dots,w(n)]\in\dal{A}{n-1}$ meaning
%that $w$ is a dominant alcove  corresponding to the minimal length coset
%representative $[w(1),\dots,w(n)]$ where $w(1)<\cdots<w(n)$.
In this special case, the level vector suffices to determine $w$,
since its window consists of the numbers in
$\{1^{\beta_1},\ldots,n^{\beta_n}\}$ ordered increasingly.
  For instance, if $w\in\mincos{A}{3}$ has
  level vector $\vec{n}(w)=(0,-2,4,-2)$, then its \bracket
  consists of the numbers in $\{1^0,2^{-2},3^{4},4^{-2}\}$ ordered
  increasingly. Using base-level notation, we have
  $(1^0,2^{-2},3^4,4^{-2})=(1,-6,19,-4)$ and thus, the base window and \bracket
  of $w$ are $[-6,-4,1,19]$ and $[2^{-2},4^{-2},1^0,3^4]$ respectively.

%\subsubsection*{Type  $C$}
A standard way to describe the affine group
$\widetilde C_n$  is by identifying it with the set of all \emph{mirrored}
$\mathbb{Z}$-permutations i.e., bijections $w$ of $\mathbb{Z}$ in itself
such that:
\begin{itemize}
\item[\rm(i)] $w(x+N)=w(x)+N$  and
\item[\rm(ii)] $w(-x)=-w(x)$ for all $x\in \mathbb{Z}$,
\end{itemize}
where $N=2n+1$ and with composition as group operation. Combining {\rm(i)} and
{\rm(ii)}, one can easily verify that $\sum_{i=1}^{2n+1}w(i)=\binom{2n+2}{2}$,
hence $\widetilde C_n$ is a subgroup of $\widetilde{A}_{2n}$.
%Clearly, each
%permutation $w\in\widetilde C_n$ is completely determined by its values
%on $\{1,\dots,n\}$.
%
%However, for reasons that will become apparent later,  we
%will identify each permutation $w\in\widetilde C_n$ with the ordered sequence
%$w=[w(1),\ldots,w(2n),w(2n+1)]$, which we will also call  the
%\emph{base window}
%\footnote{
%If we wanted to be consistent with the definition of the base window for
%$\widetilde A_{2n+1}$, the entry $w(2n+1)$ should have appeared in the
%bracket. Since, however, conditions \rm{(i)-(ii)}
%imply that $w(2n+1)=w(0)=0,$  we miss no information by ignoring $w(2n+1)$.}
%of $w$.
As we did in the type $A$ case,
we can write the abacus diagram of $w\in\widetilde C_n$  in the
concise form $[r_1^{\ell_1},\ldots,r_{2n}^{\ell_{2n}},r_{2n+1}^{\ell_{2n+1}}]$,
where now $r^\ell:=r+\ell\,N$. In this setting, {\rm(i)}
and {\rm(ii)} imply that:
\begin{enumerate}[\small i.]
\item $w(0)=0$ and $w(2n+1)=2n+1$,
%\et{}{and thus $r_{2n+1}^{\ell_{2n+1}}=(2n+1)^0$,}
\item $\{r_1,\ldots,r_{2n}\}=\{1,\ldots,2n\},$
\item $\ell_1+\cdots+\ell_{2n}=0$ and
\item if $r_i+r_j=2n+1$ then $\ell_i=-\ell_j$.
\end{enumerate}

Notice that the first condition above implies that
$r_{2n+1}=2n+1$ and $\ell_{2n+1}=0$ for every $w\in\widetilde C_n$.
Thus, the last coordinate of the abacus diagram is somewhat redundant.
In Proposition \ref{self_conj_bijec}, we rediscover and redefine type $C$ abacus
diagrams by exploiting the connection between type A and type C Shi tableaux.
These  diagrams are slightly different, in the sense that the last
coordinate mentioned above does not appear. Strictly speaking, the above 
definition of the affine group $\widetilde C_n$ is not necessary to us here. We 
refer the reader to \cite[Chapter 8.4,8.5]{bb_ccg_04} and \cite{han-bra-12} for 
more details.
\subsection{Switching from the Shi tableau to the \bracket and vice versa}
 \label{sec:shivsbr}
%\subsubsection*{Type A}
In this paragraph, we  explain how, for each dominant alcove in
$\aA_+(A_{n-1})$,
we go back and forth from its Shi tableau to its abacus diagram.
 We then  exploit this, in order to
apply the same for alcoves in $\aA_+(C_n)$ (see Proposition
\ref{self_conj_bijec}).

In \cite[Section 2.7]{fvt-fsw-11} Fishel et al., by rephrasing results of Shi
\cite{sh-pawg-99}, associate Shi tableaux with abacuses.
In particular, they prove the following two propositions.
\begin{proposition}\cite{fvt-fsw-11}
  \label{prop:shi}
  Consider a dominant alcove $w\alc_0\in\al{A}{n}$
  with \bracket  $w=[r_1^{\ell_1},\dots,r_{n+1}^{\ell_{n+1}}]$.
  Let $\mathsf{T}(w)$ be the Young tableau of shape $(n,n-1,\ldots,1)$
  whose entries  $k_{ij},\; 1\leq{i}\leq{j}\leq{n}$, are:
  \begin{equation}
  \label{1kij}
  k_{ij}=\begin{cases}\ell_{j+1}-\ell_{i},&\mbox{ if }\;\;\; r_{i}< r_{j+1}\\
  \ell_{j+1}-\ell_i-1,&\mbox{ if }\;\;\;r_{i}>r_{j+1}.\end{cases}
  \end{equation}
  The map $\mathsf T$ is a bijection between the set $\alc_+(A_n)$
   of dominant  alcoves in \shi{A_{n-1}}
   and the set of Shi tableaux of type $A_{n}$. 
\end{proposition}
%An example of the  map $\sf T$ above is shown in
%Example \ref{ex:shi}.

Before continuing with the next proposition, a few remarks are in order.
First, notice that the map above yields the same tableau if, instead of
considering the abacus of the base window, we chose any other window of
$w$ (in base-level notation). This is natural, since the tableau
${\sf T}(w)$ should be
independent of the way we represent the corresponding permutation $w$.
It is also clear from the above bijection that permutations whose \bracket
entries are sorted correspond to dominant regions
(since, in this case, all  $k_{ij}$ are positive).
Since sorted abacuses correspond to minimal length coset representatives in
$\widetilde A_{n}/A_{n}$, we verify the fact that $w\aA_0$ is a
dominant alcove in \shi{A_{n}} if and only if $w\in\mincos{A}{n}$. In the
remainder of the paper, we freely interchange
the statements \textquotedblleft$w\aA_0\in\aA_+(A_n)$\textquotedblright and
``$w\in\mincos{A}{n}$",
according to which fits better in the situation.
Thus, when we write $w\in\mincos{A}{n}$ we sometimes use it to
imply that $w\aA_0$ is a {\em dominant} alcove in \shi{A_n}
or, other times, to imply that the entries of the abacus diagram
of $w$ are {\em in increasing order}.
\smallskip

The reverse map of $\mathsf T$  is rather involved;
it sends each Shi tableau $T$ to a unique dominant alcove $w\aA_0$.
In Proposition \ref{prop:inverse},
we determine $w\aA_0$ by providing a way to
find the \emph{normalized} window of $w$,
i.e., the window whose smallest entry is $0$.
%\footnote{\et{}{There always exist such a window .. }}
Finally, in Lemma \ref{cor:inverse} we show how
to shift from the normalized to the base window or equivalently
 the abacus diagram of $w$.

\begin{proposition}\cite{fvt-fsw-11}
 \label{prop:inverse}
	Let $T=\{k_{ij}: 1\leq{i}\leq{j}\leq{n}\}$  be the Shi tableau of
	a dominant alcove in $\aA_+(A_{n})$. For each $1 \leq j \leq n$
	let
	\begin{enumerate}[(i)]
	\item
  \label{triple}
	$b_j:=\left| \{ (1,\ell,j), 1 \leq \ell < j : k_{1j} = k_{1\ell}+
	k_{\ell+1\;j}+1\}\right|$, and
	\item
	 $\sigma=[\sigma(1),\ldots,\sigma(n)]$
	be the permutation of $\{1,\dots,n\}$ with inversion table  $(b_1,\ldots,b_n)$
	\footnote{Given a permutation  $\sigma$ of  $\{1,\dots,n\}$,
	its \emph{inversion table} is a length $n$ sequence
	$b_1,\ldots,b_n$, where $b_i$ is the number of elements that are
	smaller than $i$ and appear to the right of  $i$ in  $\sigma$.}.
	\end{enumerate}
%	We define $\mathsf{B}(T)$ to be the minimal coset representative
	We define $\mathsf{B}(T)$ to be the permutation
  $w\in\mincos{A}{n-1}$
%   $w\in\widetilde A_{n-1}/A_{n-1}$
  whose normalized window has level vector
%    \footnote{
%    The level vector of the window of $w$ whose
%    smallest entry is $0^0$.
%    }
%    level vector is
\begin{equation}
(0,k_{1,\scalebox{0.5}{$\sigma(1)$}},
k_{1,\scalebox{0.5}{$\sigma(2)$}},\ldots,
k_{1,\scalebox{0.5}{$\sigma(n)$}}).
\label{equ:Psi}
\end{equation}
The map which sends the tableau $T$ to $w\aA_0$ is the reverse map of
$\mathsf T$.
\end{proposition}
%
%The next lemma shows how one can obtain the base window of $w$
%from the normalized one.

\begin{lemma}
  \label{cor:inverse}
  The base window of $w$ is obtained from its normalized window
  after subtracting  $k_{11}+\cdots+k_{1n}-1$ from each of its entries.
%
% To shift from the normalized to the base window  of $w$
% we subtract  $k_{11}+\cdots+k_{1n}-1$ from each of its entries.
%  consists of the increasing rearrangements of the numbers
%    \begin{equation}
%    \{S,1+(n+1){k_{1\sigma(1)}}+S,2+(n+1){k_{2\sigma(2)}}+S,\ldots,
%    n+(n+1){k_{1\sigma(n)}}+S\}
%    \end{equation}
%    where $S=-k_{11}-\cdots-k_{1n}+1$.
\end{lemma}
\begin{proof}
The entries of the normalized window of $w$
are the elements in $A=\{0^0,1^{k_{1\sigma(1)}},\ldots,n^{k_{1\sigma(n)}}\}$
arranged increasingly.
To shift from the normalized  to the base window,
it suffices to add an integer $S$ to every element of $A$
%the above set
%$\{0^0,1^{k_{1\sigma(1)}},\ldots,n^{k_{1\sigma(n)}}\}$
 so that the resulting set of numbers sum up to $1+2+\cdots+n+(n+1)$.
 In other words, we seek an  $S$ which  satisfies:
\begin{align}
 & (0^0+S)+(1^{k_{1\sigma(1)}}+S)+\cdots+(n^{k_{1\sigma(n)}}+S)=1+\cdots+(n+1).
 \label{eq:S}
 \end{align}
 Making use of the base-level notation,
 the left-hand side of \eqref{eq:S}  is equal to:
 \begin{align*}
  &
  0^0+1^{k_{1\sigma(1)}}+\cdots+n^{k_{1\sigma(n)}}+(n+1)S\\
 &=(n+1)^{-1}+1^{k_{1\sigma(1)}}+\cdots+n^{k_{1\sigma(n)}}+0^S \\
 &= (1+\cdots+(n+1))+(n+1)(-1+k_{1\sigma(1)}+\cdots+k_{1\sigma(n)}+S)\\
 &= (1+\cdots+(n+1))+(n+1)(-1+k_{11}+\cdots+k_{1n}+S).
  \end{align*}
 Equating the above with the right-hand side of \eqref{eq:S},
 we deduce that indeed $S=-k_{11}-\cdots-k_{1n}+1$.
\end{proof}
\medskip

In what follows, we  present an explicit example of the maps described above.
\begin{example}
 \label{ex:shi}
	The dominant alcove $w\alc_0$ with  $w=[5^{-2},2^{-1},4^0,3^1,1^2]$
  has the  Shi tableau $\sf T$ shown below
  \begin{center}
  	\begin{tikzpicture}[scale=0.46]
%    \begin{scope}[scale=0.85,xshift=0cm,yshift=0cm]
    \draw[line width= 0.5pt](0,0)rectangle(1,4);
    \draw[line width= 0.5pt](1,1)rectangle(2,4);
    \draw[line width=0.5pt](2,2)rectangle(3,4);
    \draw[line width= 0.5pt](3,3)rectangle(4,4);
    \draw[line width= 0.5pt](0,1)--(1,1);
    \draw[line width= 0.5pt] (0,2)--(2,2);
    \draw[line width= 0.5pt] (0,3)--(3,3);
    \node at(.5,3.5){$3$};\node at(1.5,3.5){$2$};\node at(2.5,3.5)
    {$1$};\node at(3.5, 3.5){$0$};
    \node at(.5,2.5){$2$};\node at(1.5,2.5){$2$};\node at(2.5,2.5){$1$};
    \node at(.5,1.5){$1$};\node at(1.5,1.5){$0$};\node at(.5,.5){$0$};
    \node at (-1,2) {${\sf T}=$};
    \end{tikzpicture}
    \end{center}
  We next illustrate the reverse map.  
  According to Proposition  \ref{prop:inverse},  in order to find the normalized
  level-vector, for each $1\leq{j}\leq{4}$, we count  the number of triples
  for which  $k_{1j}=k_{1\,\ell-1}+k_{\ell j}+1$.
  As indicated by the figure below, we have $b_4=3$, $b_3=1$ and $b_2=0$,
  ($b_1=0$ always).
  \smallskip

	\begin{center}
	\begin{tikzpicture}[scale=0.6]
%\begin{scope}[scale=0.85,xshift=0cm,yshift=0cm]
%\draw[line width= 0.5pt](0,0)rectangle(1,4);
%\draw[line width= 0.5pt](1,1)rectangle(2,4);
%\draw[line width=0.5pt](2,2)rectangle(3,4);
%\draw[line width= 0.5pt](3,3)rectangle(4,4);
%\draw[line width= 0.5pt](0,1)--(1,1);
%\draw[line width= 0.5pt] (0,2)--(2,2);
%\draw[line width= 0.5pt] (0,3)--(3,3);
%\node at(.5,3.5){$3$};\node at(1.5,3.5){$2$};\node at(2.5,3.5)
%{$1$};\node at(3.5, 3.5){$0$};
%\node at(.5,2.5){$2$};\node at(1.5,2.5){$2$};\node at(2.5,2.5){$1$};
%\node at(.5,1.5){$1$};\node at(1.5,1.5){$0$};\node at(.5,.5){$0$};
%\node at (-1,2) {${\sf T}=$};
%\end{scope}
%%%%%%%%%%%%%%%%%%%%%%%%%%%%%%%%%%%%%%

  \begin{scope}[scale=0.85,xshift=6cm,yshift=0cm]
  \draw[line width= 0.5pt](0,0)rectangle(1,4);
  \draw[line width= 0.5pt](1,1)rectangle(2,4);
  \draw[line width=0.5pt](2,2)rectangle(3,4);
  \draw[line width= 0.5pt](3,3)rectangle(4,4);
  \draw[line width= 0.5pt](0,1)--(1,1);
  \draw[line width= 0.5pt] (0,2)--(2,2);
  \draw[line width= 0.5pt] (0,3)--(3,3);
  \node at(.5,3.5){$\boldsymbol 3$};\node at(1.5,3.5){$2$};\node at(2.5,3.5)
  {$1$};\node at(3.5, 3.5){$0$};
  \node at(.5,2.5){$2$};\node at(1.5,2.5){$2$};\node at(2.5,2.5){$1$};
  \node at(.5,1.5){$1$};\node at(1.5,1.5){$0$};\node at(.5,.5){$0$};
  \node at (2,-1.8) {
    $\left.\begin{array}{cccc}
     {\boldsymbol 3}=2+0+\circled{1}\\
      {\boldsymbol 3}=1+1+\circled{1}\\
      {\boldsymbol 3}=0+2+\circled{1}
    \end{array}\right\}\Rightarrow b_4=3$
    };
  \end{scope}
%%%%%%%%%%%%%%%%%%%%%%%%%%%%%%%%%%%%%

\begin{scope}[scale=0.85,xshift=15cm,yshift=0cm]
  \draw[line width= 0.5pt](0,0)rectangle(1,4);
  \draw[line width= 0.5pt](1,1)rectangle(2,4);
  \draw[line width=0.5pt](2,2)rectangle(3,4);
  \draw[line width= 0.5pt](3,3)rectangle(4,4);
  \draw[line width= 0.5pt](0,1)--(1,1);
  \draw[line width= 0.5pt] (0,2)--(2,2);
  \draw[line width= 0.5pt] (0,3)--(3,3);
  \node at(.5,3.5){$3$};\node at(1.5,3.5){$\boldsymbol 2$};\node at(2.5,3.5)
  {$1$};\node at(3.5, 3.5){$0$};
  \node at(.5,2.5){$2$};\node at(1.5,2.5){$2$};\node at(2.5,2.5){$1$};
  \node at(.5,1.5){$1$};\node at(1.5,1.5){$0$};\node at(.5,.5){$0$};
   \node at (2,-1.8) {
     $\left.\begin{array}{l}
     {\boldsymbol 2}=2+0\\
     {\boldsymbol 2}=0+1+\circled{1}\\
     \end{array}\right\}\Rightarrow b_3=1$
    };
\end{scope}
%%%%%%%%%%%%%%%%%%%%%%%%%%%%%%%%%%%%%

\begin{scope}[scale=0.85,xshift=23cm,yshift=0cm]
\draw[line width= 0.5pt](0,0)rectangle(1,4);
\draw[line width= 0.5pt](1,1)rectangle(2,4);
\draw[line width=0.5pt](2,2)rectangle(3,4);
\draw[line width= 0.5pt](3,3)rectangle(4,4);
\draw[line width= 0.5pt](0,1)--(1,1);
\draw[line width= 0.5pt] (0,2)--(2,2);
\draw[line width= 0.5pt] (0,3)--(3,3);
\node at(.5,3.5){$3$};\node at(1.5,3.5){$2$};\node at(2.5,3.5)
{$\boldsymbol 1$};\node at(3.5, 3.5){$0$};
\node at(.5,2.5){$2$};\node at(1.5,2.5){$2$};\node at(2.5,2.5){$1$};
\node at(.5,1.5){$1$};\node at(1.5,1.5){$0$};\node at(.5,.5){$0$};
 \node at(2,-1.8){${\boldsymbol 1}=1+0 \Rightarrow b_2=0$};
\end{scope}
\end{tikzpicture}
	\end{center}
   Hence, the normalized level vector is
  $(0,k_{1\sigma(4)},k_{1\sigma(1)},k_{1\sigma(3)},k_{1\sigma(2)})=(0,3,0,2,1)$.
  This means that the elements of the set  $A=\{0^0,1^3,2^0,3^2,4^1\}$
  arranged in increasing order give the normalized window of $w$ i.e.
  $w=[0^0,2^0,4^1,3^2,1^3]$.
  In view of Corollary \ref{cor:inverse}, to
  shift from the normalized to the base window of $w$, we need  to subtract
  $k_{11}+k_{12}+k_{13}+k_{14}-1=5$  from each  of the elements of $A$.
  Using base-level notation, subtracting $5$  is equivalent to adding
  $0^{-1}$ or $5^{-2}$ to each element of $A$, from which  we deduce that the
  \bracket of  $w$ is indeed  $[5^{-2},2^{-1},4^0,3^1,1^2]$.
\end{example}
As we have seen in Section \ref{sec:viewp}, Shi tableaux of type $C_n$
coincide with self-conjugate Shi tableaux of type $A_{2n-1}$.
We  want to translate this in  terms of level vectors and
abacuses.

 Let $w$ be a  minimal length coset representative in
 $\mincos{A}{2n-1}$  with {\em antisymmetric level vector} 
 i.e. $\vec{n}(w)=(\beta_1,\ldots,\beta_n,-\beta_n,\ldots,-\beta_1)$. 
 Equivalently,  the abacus diagram of $w$ is the increasing rearrangement of
 the numbers in 
 $\{1^{\beta_1},\ldots,n^{\beta_n},
 (n+1)^{-\beta_n},\ldots,(2n)^{-\beta_1}\}$.
 We  claim that the antisymmetry of the level-vector is equivalent to the fact
 that the  abacus of $w$ is \emph{ balanced } i.e., it is of the form
 $[r_1^{\ell_1},\ldots,r_n^{\ell_n},r_{n+1}^{-\ell_{n}},\ldots,r_{2n}^{-\ell_1}]$
 where $r_i+r_{2n+1-i}=2n+1$. This is the content of the next lemma.

 \begin{lemma}
   \label{small_lemma}
   Let $w\aA_0$ be a dominant alcove in \shi{A_{2n-1}}.
   Then $\vec{n}(w)=(\beta_1,\ldots,\beta_n,-\beta_n,\ldots,-\beta_1)$
    if and only if the abacus diagram of $w$ is balanced.
   \end{lemma}
   \begin{proof}
     To prove the forward direction, we assume that the level vector is 
     antisymmetric. Then, recalling  that 
     $\alpha^{\ell}:=\alpha+\ell\,2n$, we 
     have:
  \begin{align*}
  k^{\beta_k}<\lambda^{\beta_{\lambda}}\Leftrightarrow\, &
  (2n+1-k)^{-\beta_k}<(2n+1-\lambda)^{-\beta_{\lambda}} \;\;\;\; \mbox{and}
  \\
  (2n+1-k)^{-\beta_k}<\lambda^{\beta_{\lambda}}\Leftrightarrow\, &
  (2n+1-\lambda)^{-\beta_{\lambda}}<k^{\beta_k}.
  \end{align*}   
  The above   equivalences  imply that the elements which precede 
  $k^{\beta_k}$  in the abacus diagram of $w$ are as many as the elements which 
  succeed
  $(2n+1-k)^{-\beta_k}$. This further implies that
  the diagram is balanced. 
  The reverse direction is immediate from the definition of balanced. 
 \end{proof}

\begin{proposition}
  \label{self_conj_bijec}
  The bijection $\mathsf T$ of Proposition \ref{prop:shi}
  restricts to one between self-conjugate tableau of alcoves
  in $\aA_+(A_{2n-1})$ and minimal length coset representatives
  $w\in\mincos{A}{2n-1}$ whose level vector is antisymmetric
  i.e., $\beta_i=-\beta_{2n+1-i}$ for all $1\leq{i}\leq{n}$.
\end{proposition}
\begin{proof}
Let  $w\in\mincos{A}{2n-1}$
with antisymmetric level vector.
In view of Lemma \ref{small_lemma},
 $w$ has balanced abacus diagram i.e., it can be written as
$[r_1^{\ell_1},\ldots,r_n^{\ell_n},r_{n+1}^{-\ell_{n}},\ldots,r_{2n}^{-\ell_1}]$.
 Thus, when applying \eqref{1kij}
we have
  \begin{align*}
  k_{2n-j\;2n-i}&=
  \begin{cases}
  \ell_{2n-i+1}-\ell_{2n-j},&\mbox{ if }\;\;\; r_{2n-j}< r_{2n-i+1}\\
  \ell_{2n-i+1}-\ell_{2n-j}-1,&\mbox{ if }\;\;\;r_{2n-j}>r_{2n-i+1}
  \end{cases}\\
  &=
  \begin{cases}
  -\ell_{i}-(-\ell_{j+1}),&\hspace{0.2cm}\mbox{ if}
  \;\;\;2n+1-r_{j+1}<2n+1-r_{i}\\
  -\ell_{i}-(-\ell_{j+1})-1,&\hspace{0.2cm}\mbox{ if } \;\;\;2n+1-r_{j+1}>
  2n+1-r_{i}
  \end{cases}\\
  &=
  \begin{cases}
  \ell_{j+1}-\ell_{i},&\hspace{0.9cm}\mbox{ if }\;\;\; r_{i}< r_{j+1}\\
  \ell_{j+1}-\ell_i-1,&\hspace{0.9cm}\mbox{ if }\;\;\;r_{i}>r_{j+1}
  \end{cases}\\
  &=k_{ij},
  \end{align*}
  which means that  ${\sf T}(w)$ is self-conjugate.

  To prove the reverse, we use induction on $n$, the claim for $A_1$ being 
  trivial. Next, we assume that our claim holds for the group $A_{2n-3}$ and we 
  prove it for $A_{2n-1}$. To this end, consider a self-conjugate tableau $T$ 
  of  an alcove $w\aA_0\in\aA_+(A_{2n-1})$. Let $T'$ be the tableau we obtain 
  from $T$ after  deleting its first row and column. Clearly, $T'$ is a 
  self-conjugate Shi tableau of an alcove $w'\aA_0$ in $\aA_+(A_{2n-3})$ and 
  induction implies that we can write its abacus diagram as
  $w'=[r_2^{\ell_2},\ldots,{r}_{n}^{\ell_{n}},
  {r}_{n+1}^{-\ell_n},\ldots,{r}_{2n-1}^{-\ell_2}]$, where
  $\{{r}_2,\ldots,{r}_{2n-2}\}=\{1,\ldots,2n-1\}$ and 
  ${r}_i+{r}_{2n+1-i}=2n-1$. 
  By an appropriate shifting of the $r_i$'s 
  \footnote{This claim is somewhat subtle. Induction implies that 
  $\{r_2,\ldots,r_{2n-2}\}=\{1,\ldots,2n-1\}$ with $r_i+r_{2n+1-i}=2n-1$
    for all $2\leq{i}\leq{n-1}$. Then, depending on the value of $r_1$, there 
    is a unique way to shift the above pairs of $r_i$'s, either by  
    $r_i\leftarrow{}r_{i}+1$  and  $r_{2n+1-i}\leftarrow{}r_{2n+1-i}+1$  \,or\, 
    by $r_i\leftarrow r_{i}$  and $r_{2n+1-i}\leftarrow{}r_{2n+1-i}+2$,  so 
    that 
    $\{r_1,\ldots,r_{2n}\}=\{1,\ldots,2n+1\}$, $r_i+r_{2n+1-i}=2n+1$ for all 
    $1\leq{i}\leq{n}$ and again ${\sf T}(w')=T'$.
    }, we can  assume that 
  $w=[r_1^{\ell_1},r_2^{\ell_2},\ldots,r_{n}^{\ell_{n}},
  r_{n+1}^{-\ell_n},\ldots,r_{2n-1}^{-\ell_2},r_{2n}^{\ell_{2n}}]$,
  where $r_i+r_{2n+1-i}=2n+1$ for all $2\leq{i}\leq{n-1}$,
  i.e., that the abacus of $w$ is balanced
  except possibly from the pair consisting of its first and last entry.
  Since the entries of the base window sum up to $1+\cdots+2n=n(2n+1)$,
  we deduce that
  \begin{align}
  r_1^{\ell_1}+r_{2n}^{\ell_{2n}}=2n+1.
  \label{r1r2n}
  \end{align}
  Our goal is to show that  $r_1+r_{2n}=2n+1$ and $\ell_1+\ell_{2n}=0$.
  To this end, for $2\leq{i}\leq{2n}$ we set $\epsilon_i=1$ if $r_1>r_i$ and
  $\epsilon_i=0$ otherwise. In view of \eqref{1kij},
  the sum of the entries of the first row of $T$ is
  \begin{align}
  R_1 = \sum_{i=1}^{2n}k_{1i}& = (\ell_{2n}-\ell_1-\epsilon_{2n})
  + \sum_{i=2}^n (-\ell_i-\ell_1-\epsilon_i)
  + \sum_{i=1}^{n-1} (\ell_i-\ell_1-\epsilon_{i+n})\notag\\
  %\label{R1a}\\
  &=(\ell_{2n}-\ell_1)-2(n-1)\ell_1-(r_1-1)
  \label{R1b}
  \end{align}
  where, in the last equation,  we used the fact that
  $\sum_{i}\epsilon_i=|\{i: r_i<r_1\}|=r_1-1$.
  Analogously, for $1\leq{i}\leq{2n-1}$ we set
  $\epsilon'_i=1$ if $r_i > r_{2n}$ and $\epsilon'_i=0$ otherwise.
  In view of \eqref{1kij},
  the sum of the entries of the first column  of $T$ is
  \begin{align}
  C_1 = \sum_{j=1}^{2n}k_{j\,2n}& = (\ell_{2n}-\ell_1-\epsilon'_1)
  + \sum_{j=2}^n (\ell_{2n}-\ell_j-\epsilon'_i)
  + \sum_{j=2}^n (\ell_{2n}-\ell_j-\epsilon'_{i+n})\notag\\
  &=(\ell_{2n}-\ell_1)-2(n-1)\ell_{2n}-(2n-r_{2n})
  \label{C1b}
  \end{align}
  where again, in the last equation,  we used the fact
  that $\sum_{j}\epsilon'_j=|\{j:r_j>r_{2n}\}|=2n-r_{2n}$.
  Since the tableau $T$ is self-conjugate we have that  $R_1=C_1$
  and thus, using base-level notation for the expressions  \eqref{R1b} and
  \eqref{C1b}, we conclude that
  $r_1^{\ell_1}+r_{2n}^{\ell_{2n}}=2\ell_1+2\ell_{2n}+1+2n$.
  In view of \eqref{r1r2n}, this implies that $\ell_1+\ell_{2n}=0$
  and thus $r_1+r_{2n}=2n+1$.
\end{proof}

As we have already pointed out, in Proposition \ref{self_conj_bijec} we recover
a slightly modified description of  the affine group $\widetilde C_n$,
compared to that given in Section \ref{sec:1line}.
The difference is that, in the above proposition,
 we no longer require $w(0)=0$ and thus all
symmetry is encoded  in an abacus diagram of size $2n$
i.e., in a subgroup of $\widetilde A_{2n-1}$.
This modification arises naturally, since
the images through the map $\mathsf T$ are
tableaux of the required size and symmetry.
Thus, from now on we identify dominant alcoves in \shi{C_n}
with minimal length coset representatives $w\in\mincos{A}{2n-1}$
whose abacus diagram is balanced i.e.,
$w=[r_1^{\ell_1},\ldots,r_n^{\ell_n},r_{n+1}^{-\ell_{n}},\ldots,r_{2n}^{-\ell_1}]$
 and $r_i+r_{2n+1-i}=2n+1$ for all $i$.

\subsection{Criteria for $m$-minimality}
\label{m-min-criteria}

As we mentioned in Section \ref{sec:mShi}, each dominant region in
\mshi{\Phi} is uniquely represented by its $m$-minimal alcove.
This allows us to identify the set of regions in $\rR_+^m(\Phi)$ with that
of $m$-minimal alcoves in $\aA_+(\Phi)$, for which we can apply  the
bijections of  Section \ref{sec:shivsbr} (involving  Shi tableaux and
abacus diagrams). However, to do this, we need to distinguish $m$-minimal
among all alcoves in $\aA_+(\Phi)$. Proposition \ref{self_conj_bijec} implies
that it suffices to formulate criteria for $m$-minimality
in the type $A$ case, since
$m$-minimal alcoves in \shi{C_n} correspond to
$m$-minimal alcoves in \shi{A_{2n-1}} having self-conjugate Shi tableau.

The following theorem is equivalent to \cite[Theorem 4.5]{ahj-rcosc-13}, where
criteria for $m$-minimality are expressed in terms of flush abacuses.
To make the paper self-contained, we prove it using our own setup.
\begin{theorem}
 \label{thm:mconditions}
	Let $\alc=w\alc_0\in \alc_+(A_{n-1})$ be a dominant alcove with level vector
   $\vec{n}(w)=(\beta_1,\ldots,\beta_{n}).$
	Then,  $\alc$ is the $m$-minimal alcove of a dominant region
	$\rR$  in \mshi{A_{n-1}} if and only if:
	\begin{itemize}
	\item[\rm(i)] $\beta_{i+1}-\beta_i\leq  m $ for all $1\leq i\leq n-1$, and
	\item[\rm(ii)] $\beta_1-\beta_{n}-1\leq m.$
	\end{itemize}
\end{theorem}
In the current setup, it its easier to prove Theorem \ref{thm:mconditions2},
which is a  reformulation of Theorem \ref{thm:mconditions}.
In the proof, we repeatedly use Corollary \ref{cor:same_region} of the
Appendix, which allows us to compare the face defining inequalities of
a region in \mshi{\Phi} with those of its $m$-minimal alcove.
\begin{theorem}
   \label{thm:mconditions2}
	Let $\alc=w\alc_0\in \alc_+(A_{n-1})$ be a dominant alcove with \bracket
	$w=[r_1^{\ell_1},\ldots,r_{n}^{\ell_{n}}].$
	Then, $\alc$ is the $m$-minimal alcove of a dominant region
	$\mathcal R\in$ \mshi{A_{n-1}} if and only if:
	\begin{itemize}
    \item[\rm(i)] $\ell_j-\ell_{i}\leq{m}$ for all 	$1\leq{i}\leq{j}\leq{n}$
    with     $r_j=r_{i}+1\leq{n}$,\, and
 	\item[\rm(ii)] $\ell_j-\ell_{i}-1\leq m$ if $r_j=1$ and $r_{i}=n.$
	\end{itemize}
\end{theorem}

\begin{proof}
	Let $\alc=w\alc_0$  be the $m$-minimal alcove of a dominant region $\rR$ in
	\mshi{A_{n-1}} with	 \bracket $w=[r_1^{\ell_1},\ldots,r_n^{\ell_n}]$
	and assume that one of the conditions in the statement of the lemma is
	violated. We separate cases:

  {\rm(i)}  If there exist $r_{\nu},r_{\mu}$ with
	$r_{\nu}=r_{\mu}+1\leq{}n$ for which $\ell_{\nu}-\ell_{\mu}>m$
	then, since the entries in the \bracket diagram are sorted,	we can write
	$w=[r_1^{\ell_1},\ldots,r_{\mu}^{\ell_{\mu}},\ldots,r_{\nu}^{\ell_{\nu}},
	\ldots,r_n^{\ell_n}]$. Consider  the affine permutation
	$w'=[r_1^{\ell_1},\ldots,r_{\nu}^{\ell_{\mu}},\ldots,r_{\mu}^{\ell_{\nu}},
	\ldots,r_n^{\ell_n}]$ where we have exchanged $r_{\nu}$ and $r_{\mu}$
   without their levels, and notice that $w'$ is still sorted.
  Thus, $w'\alc_0$ is a dominant alcove.
	Moreover, from \eqref{1kij} one can check that  its Shi tableau
	$\{k'_{ij}:1\leq{}i\leq{}j\leq{}n-1\}$ has all
	$k_{ij}=k'_{ij}$ except from the entry $k'_{\mu\,\nu-1}$
	for which $k'_{\mu\,\nu-1}=\ell_{\nu}-\ell_{\mu}-1=k_{\mu\,\nu-1}-1$.
  Since we have assumed that $\ell_{\nu}-\ell_{\mu}-1\geq{m}$,
  Corollary \ref{cor:same_region} implies that $w'\mathcal A_0$
	and $w\mathcal A_0$ lie in the same region $\mathcal R$.
  %%%%%

	{\rm(ii)} If $r_{\mu}=n$, $r_{\nu}=1$ and
	$\ell_{\nu}-\ell_{\mu}-1>m$ then, since the entries in the bracket
	are sorted, we can write
		$w=[r_1^{\ell_1},\ldots,n^{\ell_{\mu}},\ldots,1^{\ell_{\nu}},
		\ldots,r_n^{\ell_n}]$. Next, notice that for all $r_i^{\ell_i}$'s
		between $n^{\ell_{\mu}}$ and $1^{\ell_{\nu}}$ it is
		$\ell_{\mu}<\ell_i<\ell_{\nu}$ (otherwise, the entries of the bracket
		would not be in increasing order). Thus, the bracket of
    $w'=[r_1^{\ell_1},\ldots,1^{\ell_{\mu}+1},\ldots,n^{\ell_{\nu}-1},
		\ldots,r_n^{\ell_n}]$ is also sorted and hence
		$w'\alc_0$ is a dominant alcove.
		Moreover, its Shi tableau $\{k'_{ij}:1\leq{}i\leq{}j\leq{}n-1\}$
			has all $k_{ij}=k'_{ij}$ except for $k'_{\mu\;\nu-1}$ for
			which $k'_{\mu\;\nu-1}=k_{\mu\;\nu-1}-1=\ell_{\nu}-\ell_{\mu}-2\geq{m}.$
			As before, Corollary \ref{cor:same_region} implies that
      both $w'\alc_0$ and $w\alc_0$ lie in the same region $\rR$.

	In both cases we have found an alcove $w'\alc_0$ in $\rR$
	having all Shi coordinates equal to those of $w\alc_0$ except for one
	which is smaller. This  contradicts the fact
	that $w\alc_0$ is the minimal alcove of $\rR$.
  \smallskip

  	To prove the reverse, we assume that  the conditions in the statement of
  	the theorem  hold for some $w$ and we prove that all alcoves adjacent to
  $w\alc_0$  i.e., all 	those  which share a facet with $w\alc_0$, either lie
  in another region or have larger coordinates. This immediately implies that
  $w\alc_0$ is 	$m$-minimal.

  Two alcoves $w\aA_0,w'\aA_0$ are adjacent if they have all their Shi 
  coordinates the same, except for one in which they differ by $\pm 1$. 
  Arguing as   above, this can only happen   when, in the abacus digram of $w$, 
  we either exchange  pairs $r_{\mu},r_{\nu}$ with $r_{\nu}=r_{\mu}+1\leq{n}$ 
  or the pair $n$,$1$
  with appropriate alteration of their levels.
  This indicates the following four cases, in which
  $w\aA_0$ and $w'\aA_0$ differ by exactly one Shi
  coordinate. \\
  {\em Case 1:} If $r_{\nu}=r_{\mu}+1\leq{n}$ then,  from our assumption,
  $\ell_{\nu}-\ell_{\mu}\leq{m}$. We further distinguish cases.\\
%  $\bullet$
  If $\ell_{\mu}\leq\ell_{\nu}$  then we switch from
  $w=[r_1^{\ell_1},\ldots,r_{\mu}^{\ell_{\mu}},\ldots,r_{\nu}^{\ell_{\nu}},
  \ldots,r_n^{\ell_n}]$ to  $w'=$
  $
  [r_1^{\ell_1},\ldots,r_{\nu}^{\ell_{\mu}},\ldots,r_{\mu}^{\ell_{\nu}},
    \ldots,r_n^{\ell_n}]$.
    The only different Shi coordinate of the acloves $w\aA_0$ and $w'\aA_0$
    is  $k_{\mu\;\nu-1}=\ell_{\nu}-\ell_{\mu}\leq{m} $
    which becomes $ k'_{\mu\; \nu-1}=\ell_{\nu}-\ell_{\mu}-1<m$.
   In view of Corollary \ref{cor:same_region} we deduce that $w'\aA_0$ and
   $w\aA_0$ lie in  different regions.\\
%   $\bullet$
    If $\ell_{\nu}<\ell_{\mu}$ then we switch from $
   [r_1^{\ell_1},\ldots,r_{\nu}^{\ell_{\nu}},\ldots,r_{\mu}^{\ell_{\mu}},
 \ldots,r_n^{\ell_n}]$ to 
 $ [r_1^{\ell_1},\ldots,r_{\mu}^{\ell_{\nu}},\ldots,r_{\nu}^{\ell_{\mu}},
 \ldots,r_n^{\ell_n}] $.
 The only different Shi coordinate of the acloves $w\aA_0$ and $w'\aA_0$
 is $ k_{\mu\;\nu-1}=\ell_{\mu}-\ell_{\nu}-1$ which becomes
 $k'_{\mu\;\nu-1}=\ell_{\mu}-\ell_{\nu}$.  Since
 $k'_{\mu\;\nu-1}>k_{\mu\;\nu-1}$, the
 alcove $w'\aA_0$ does not lie closer to the origin than $w\aA_0$. \\
  {\em Case 2:} If $r_{\nu}=1 $ and $r_{\mu}=n $
     then,  from our assumption,  $\ell_{\nu}-\ell_{\mu}-1\leq{m}$.
     We further distinguish cases. \\
     If $\ell_{\mu}+1<\ell_{\nu}$ then,  we switch from the abacus diagram 
     $w=[r_1^{\ell_1},\ldots,n^{\ell_{\mu}},\ldots,1^{\ell_{\nu}},
     \ldots,r_n^{\ell_n}]$ to 
     $w'=[r_1^{\ell_1},\ldots,1^{\ell_{\mu}+1},\ldots,n^{\ell_{\nu}-1},
     \ldots,r_n^{\ell_n}]$.
     The only different Shi coordinate in the acloves $w\aA_0$ and $w'\aA_0$
     is  $k_{\mu\;\nu-1}=\ell_{\nu}-\ell_{\mu}-1\leq{m}$
     which  becomes $ k'_{\mu\;\nu-1}=(\ell_{\nu}-1)-(\ell_{\mu}+1)
     =\ell_{\nu}-\ell_{\mu}-2<m$.  In view of Corollary \ref{cor:same_region},
     $w'\aA_0$
     and $w\aA_0$ lie in different  regions.\\
%     $\bullet$
     If $\ell_{\mu}+1\geq \ell_{\nu}$
     then, we switch from the abacus diagram 
  $w=[r_1^{\ell_1},\ldots,1^{\ell_{\nu}},\ldots,n^{\ell_{\mu}},
  \ldots,r_n^{\ell_n}]$ to 
  $w'=[r_1^{\ell_1},\ldots,n^{\ell_{\nu}-1},\ldots,1^{\ell_{\mu}+1},
  \ldots,r_n^{\ell_n}]$.
  The only different Shi coordinate in the alcoves $w\aA_0$ and $w'\aA_0$
  is  $k_{\mu\;\nu-1}=\ell_{\mu}-\ell_{\nu}$
  which becomes $k'_{\mu\;\nu-1}=(\ell_{\mu}+1)-(\ell_{\nu}-1)-1=
  \ell_{\mu}-\ell_{\nu}+1$. Since $k'_{\mu\;\nu-1}>k_{\mu\;\nu-1}$, the
   alcove $w'\aA_0$ does not lie closer to the origin than $w\aA_0$.
\end{proof}

\subsection{The bijection ${\sf Bj}_1$}
\label{sec:proj}

We now have all the ingredients to describe the map ${\sf Bj}_1$ and prove that
it is a bijection.
Let $[w(1),\dots,w(n)]$ be the base window of some
$w\in\mincos{A}{n-1}$. For each $k\in \mathbb{Z}$
we define the \emph{ $k$-shift of $w$}, denoted by
$\w{k}$,  as the ordered collection of integers:
\begin{equation}
\label{equ:wk}
\w{k}:=(w(1)+k,\dots,w(n)+k).
\end{equation}
%  Note that $\w{k}$ does not necessarily correspond to a window of $w$.
  As will be shown below, there is a natural way to ``expand" $\w{k}$
  in order to   obtain   minimal length coset representatives in
  $\widetilde{A}_{2n-1}/A_{2n-1}$. To
  describe our expansion, let $r_i:=(w(i)+k)\!\!\!\mod{n}$ be the base,
  $\ell_i:=\fl{\tfrac{w(i)+k}{n}}$ be the  level of each  $w(i)+k$ and  rewrite
  \eqref{equ:wk} as:
\begin{equation}
 \w{k}=[r_1^{\ell_1},\ldots,r_n^{\ell_n}].
\end{equation}
Next,  notice that $\{r_1,\ldots,r_n\}=\{1,\ldots,n\}$.
Indeed, condition  \eqref{equ:bracket}  implies that,  for each
base window it is
$\{w(1)\!\!\!\mod n,\ldots,w(n)\!\!\!\mod n\}=\{1,\ldots,n\}$,
hence also $\{(w(1)+k)\!\!\!\mod n,\ldots,$$\,(w(n)+k)\!\!\!\mod
n\}$$=\{1,\ldots,n\}$.
This allows us to define the level vector of $\w{k}$
as we did in the case of base windows,
% minimal length coset representatives,
regardless of the fact that the levels do not sum up to zero.
Now, we are in a position to formulate the definition
of our expansion.
\begin{definition}
 \label{def:antisC}
     	Let $w$ be an ordered collection of integers
     	$[r_1^{a_1},\ldots,r_n^{a_n}]$
	such that $\{r_1,\ldots,r_n\}=\{1,\ldots,n\}$.
	The \emph{antisymmetric expansion $\alpha(w)$ of $w$},
	 is  the sorted array  $[\bar{w}(1),\ldots,\bar{w}(2n)],$
   with elements:
   \begin{equation}
   \label{equ:C}
   \{\bar{w}(1)<\bar{w}(2)<\ldots<\bar{w}(2n)\}=
   \{  r_i^{a_i},(2n+1-r_i)^{-a_i} \mbox{ for } 1\leq{i}\leq{n}\}.
   \end{equation}
%     The \emph{type $B$ antisymmetric expansion $\Bexp{w}{}$ of $w$}, is
%  the sorted bracket:
%	\[ \Bexp{w}{}:=[\bar{w}(1),\ldots,\bar{w}(2n+1)], \;\;\;\;\;\mbox{ where }\]
%% 	  where
% 	  \begin{equation}
% 	  	 \label{equ:B1}
% 	  	   \{\bar{w}(1)<\bar{w}(2)<\ldots<\bar{w}(2n+1)\}=
% 	  	 \{ r_1^{\ell_1},\ldots,r_n^{\ell_n},(n+1)^0,
% 	  	 (2n+2-r_n)^{-\ell_n},\ldots,(2n+2-r_1)^{-\ell_1}
% 	  	\}.
% 	   	  \end{equation}
\end{definition}
\bigskip
%
%Next, we write  $\w{k}$ as in ....
%\begin{equation}
%\w{k}=[r_1^{\ell_1},\ldots,r_n^{\ell_n}],
%\end{equation}
%where $w(i)+k=r_i+n\ell_i$ and notice that
%$\{r_1,\ldots,r_n\}=\{1,\ldots,n\}$.
%Indeed, in view of \eqref{equ:bracket}
%we have $\{w(1)\!\!\!\mod n,\ldots,w(n)\!\!\!\mod n\}=\{1,\ldots,n\}$,
%and therefore $\{(w(1)+k)\!\!\!\mod n,\ldots\,(w(n)+k)\!\!\!\mod
%n\}=\{1,\ldots,n\}$.
%This allows us to define the vector of levels of $\w{k}$
%as we did in the case of minimal coset representatives.
%

The next lemma shows that for each $w\in\mincos{A}{n-1}$ and $k\in\mathbb{Z}$, 
the antisymmetric expansion of its $k$-shift produces elements of 
$\mincos{A}{2n-1}$ with balanced abacus.
\begin{lemma}
\label{lem:antiC}
  Let $w\in\mincos{A}{n}$ with level vector
  $(\beta_1,\ldots,\beta_n)$. Then, for each $k\in\mathbb{Z}$,  the
  antisymmetric expansion  $\alpha({\w{k}})$ of its $k$-shift, is a minimal
  length  coset representative in $\mincos{A}{2n-1}$
  whose abacus diagram is balanced.
\end{lemma}
\begin{proof}
First, notice that if $\{r_1,\ldots,r_n\}=\{1,\ldots,n\}$
then $\{2n+1-r_1,\ldots,2n+1-r_n,r_1,\ldots,r_n\}=\{1,\ldots,2n\}$.
Also, by the definition of the antisymmetric expansion, the
levels in $\alpha(\w{k})$ sum up to $0$. Hence, both conditions in
\eqref{equ:bracket} are satisfied and thus $\alpha(\w{k})$ in an element of
$\widetilde A_{2n-1}$. Finally, by construction, the level vector of $\w{k}$
is antisymmetric and thus, in view of Lemma \ref{small_lemma}, its abacus is
balanced.
\end{proof}
The first step towards \bj\ is the following proposition, which shows how, from 
any alcove in $\aA_+(A_{n-1})$ and integer $k\in\mathbb{Z}$ we can uniquely 
define an alcove in $\aA_+(C_n)$. The forward direction of the bijection 
follows naturally from the $k$-shift. The reverse, which extracts   from a 
self-conjugate tableau of type $A_{2n-1}$, a type $A_{n-1}$ tableau 
appropriately shifted, is more complicated. 

\begin{proposition}
\label{prop:proj alc}
	The map $\psi:\mathcal{A}_+(A_{n-1})\times\mathbb{Z}\longrightarrow$
	$\mathcal{A}_+(C_n)$
	sending each pair $(w,k)$ to $\alpha(\w{k})$
	is a 	bijection.  Its 	inverse
	map sends each $\bar{w}$ to $(w,k)$ where:
	\begin{enumerate}[(i)]
	\item 	 $k$ is the sum of the lower-half of the
	level vector
	 $\vec{n}(\bar{w})=(\beta_1,\ldots,\beta_n,-\beta_n,\ldots,-\beta_1)$
	of $\bar{w}$, and
	\item  $w$ is the minimal length coset representative in
  $\widetilde A_{n-1}/A_{n-1}$	 with level vector
	  $$\vec{n}(w)=(\beta_{n-r+1}+\ell+1,\ldots,\beta_n+\ell+1,\beta_1+\ell,
	\ldots,\beta_{n-r}+\ell),$$
	 where $r=-k\!\!\!\mod{n}$ and
	$\ell=-\frac{k+r}{n}$.
	\end{enumerate}

\end{proposition}

\begin{proof}
 Lemma \ref{lem:antiC} ensures that the map in the statement of the
 proposition is well defined.
For the reverse,
% we backtrack the steps of the antisymmetric expansion
% of the $k$-shift of $w$. To this end,
 let $\bar{w}\alc_0$ be an alcove
 in $\aA_+(C_n)$ with level vector
 $\vec{n}(\bar{w})=(\beta_1,\ldots,\beta_n,-\beta_n,\ldots,-\beta_1)$.
 Reversing the procedure of  the antisymmetric expansion, we see that
 $\w{k}$ consists of the numbers $ \{1^{\beta_1},\ldots,n^{\beta_n}\}$.
 In order to shift from $\w{k}$ to its base window $w$,
 we have to subtract an integer $k$ form each $i^{\beta_i}$,
 so that the resulting set of numbers sum up to $1+\cdots+n$.
To determine $k$, we set $r:=-k\!\!\!\mod{n}$ and
  $\ell:=\frac{-r-k}{n}$, so that  in the base-level notation $-k=r^{\ell}$
  and we ask that:
  \begin{align}
   1+\cdots+n & = 1^{\beta_1}-k+\cdots+n^{\beta_n}-k
               \notag\\
              & = (1^{\beta_1}+r^{\ell})+\cdots+(n^{\beta_n}+r^{\ell})
               \notag\\
              &=(1+r)^{\beta_1+\ell}+(2+r)^{\beta_2+\ell}+\cdots+n^{\beta_{n-r}+\ell}
              +(n+1)^{\beta_{n-r+1}+\ell}+\cdots+(n+r)^{\beta_{n}+\ell}
               \notag\\
              &= (1+r)^{\beta_1+\ell}+(2+r)^{\beta_2+\ell}+\cdots+n^{\beta_{n-r}+\ell}
              +1^{\beta_{n-r+1}+\ell+1}+\cdots+r^{\beta_{n}+\ell+1}.
              \label{b}
  \end{align}
  The above equalities hold if  the sum of the levels of the numbers  in
  \eqref{b} is  $0$
 i.e., if $(\beta_1+\cdots+\beta_n)+n\ell+r=0$.
 The last is equivalent to  $k=\beta_1+\cdots+\beta_n$,
 which proves {\em (i)}. 
 It is also evident from \eqref{b} that the level vector of $w$
 is indeed the one described in item  {\em (ii)} of the proposition.
\end{proof}
\medskip

\begin{figure}[h!]
  \begin{center}
   \begin{tikzpicture}[scale=0.5]
  \begin{scope}[xshift=-2 cm,yshift=2cm]
  \node at (2,-2) {\footnotesize
    $[4^{-3},1^{-1},2^2,3^2]$\;\;};
  \foreach \j in {0,...,2} {\foreach \i in {0,...,\j} { \draw (\i,\j)
      rectangle (1+\i,\j+1); } }
  \node at (.5, 2.5) {4}; \node at (1.5, 2.5) {4}; \node at (2.5, 2.5){1};
  \node at (.5, 1.5) {3}; \node at (1.5, 1.5) {3};
  \node at (.5,.5)   {0};
  \end{scope}
  %%%%%%%%%%%%%%%%%%%%%%%%%%%%%%%%%%%%%%%%%%%%%%%%%%%%%%%%%%%%%%%%%%%
  %%%%%%%%%%%%%%%%%%%%%%%%%%%%%%%%%%%%%%%%%%%%%%%%%%%%%%%%%%%%%%%%%%%
  \begin{scope}[xshift=7 cm,yshift=0cm]
  \foreach \j in {0,...,6} {\foreach \i in {0,...,\j} { \draw (\i,\j)
      rectangle (1+\i,\j+1); } }
  \node at (.5, 6.5){6};\node at(1.5,6.5){5};\node at(2.5,6.5){\mk{4}};\node
  at(3.5,6.5){\mk{4}};\node at(4.5,6.5){2};\node at(5.5,6.5){1};\node
  at(6.5,6.5){\mk{1}};
  \node at (.5, 5.5) {5}; \node at (1.5, 5.5) {4}; \node at (2.5, 5.5)
  {\mk{3}}; \node at (3.5,5.5) {\mk{3}};\node at (4.5,5.5) {1};\node at
  (5.5,5.5) {0};
  \node at (.5, 4.5) {4}; \node at (1.5, 4.5) {3}; \node at (2.5, 4.5) {3};
  \node at (3.5,4.5) {2}; \node at (4.5,4.5)  {0};
  \node at (.5, 3.5) {4}; \node at (1.5, 3.5) {3}; \node at (2.5, 3.5) {2};
  \node at (3.5, 3.5){1};
  \node at (.5, 2.5) {2}; \node at (1.5, 2.5) {1}; \node at (2.5, 2.5)
  {\mk{0}};
  \node at (.5, 1.5) {1}; \node at (1.5, 1.5) {0};
  \node at (.5,.5)   {1};
  \node[text width=6cm] at (3.5,-2.5) {\footnotesize
%    $[4^{-3},1^{-1},2^2,3^2]\overset{k_6=2^{-1}}{\Rightarrow}
    $[4^{-3},1^{-1},2^2,3^2]\overset{k=2^{-1}}{\Rightarrow}
    [2^{-3},3^{-2},4^1,1^2]$
    \bigskip

    $[\mk{2^{-3}},\mk{3^{-2}},8^{-2},5^{-1},\mk{4^1},\mk{1^2},6^2,7^3]$\;\;};
  \end{scope}

  \begin{scope}[xshift=19 cm,yshift=0cm]
  \foreach \j in {0,...,6} {\foreach \i in {0,...,\j} { \draw (\i,\j)
      rectangle (1+\i,\j+1); } }
  \node at(.5, 6.5){6};\node at(1.5,6.5){5};\node at(2.5,6.5){\mk{4}};
  \node at(3.5,6.5){\mk{4}};\node at(4.5,6.5){1};\node at(5.5,6.5){1};
  \node at(6.5,6.5){\mk{1}};
  \node at(.5, 5.5){5};\node at(1.5, 5.5){4};\node at(2.5,
  5.5){\mk{3}};
  \node at(3.5,5.5){\mk{3}};\node at(4.5,5.5){0};\node at(5.5,5.5){0};
  \node at(.5,4.5){4};\node at(1.5,4.5){3};\node at(2.5,
  4.5){3};
  \node at(3.5,4.5){3};\node at(4.5,4.5){0};
  \node at(.5,3.5){4};\node at(1.5,3.5){3};\node at(2.5,3.5){3};
  \node at(3.5,3.5){3};
  \node at(.5,2.5){1};\node at(1.5,2.5){0};\node at(2.5,2.5){\mk{0}};
  \node at(.5,1.5){1};\node at(1.5,1.5){0};
  \node at(.5,.5){1};
  \node[text width=6cm] at (4.5,-2.5) {\footnotesize
    $[4^{-3},1^{-1},2^2,3^2]\overset{k=-1^{0}}{\Rightarrow}
    [3^{-3},4^{-2},1^2,2^2]$
    \bigskip

    $[\mk{3^{-3}},\mk{4^{-2}},7^{-2},8^{-2},\mk{1^2},\mk{2^2},5^2,6^3]$};
  \end{scope}
  \end{tikzpicture}
  \caption{Two examples of the map $\psi$ of Proposition~\ref{prop:proj alc}}
  \label{fig:two_expansions}
\end{center}
\end{figure}
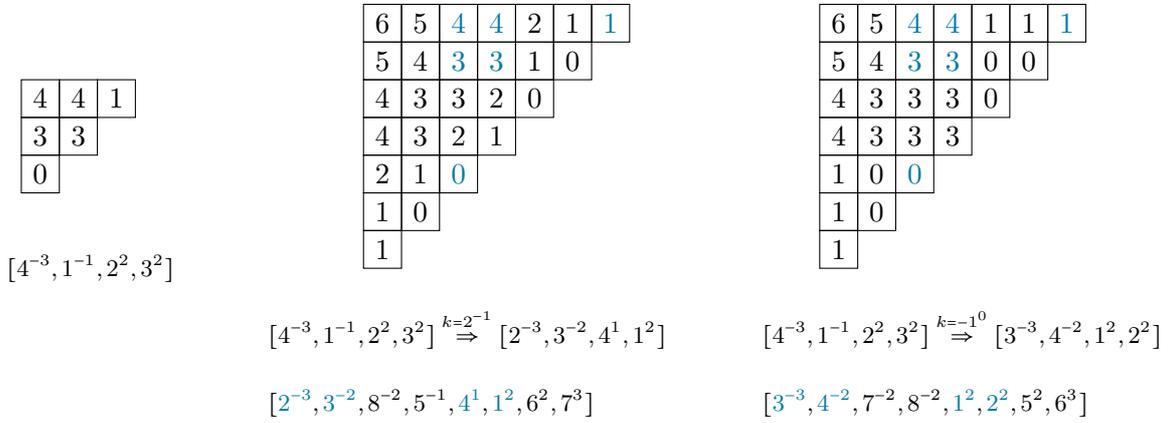
In Figure \ref{fig:two_expansions} we illustrate two instances of the map
$\psi$ of Proposition \ref{prop:proj alc}.
The tableau on the left corresponds to the
3-minimal alcove $w\aA_0\in\aA_+(A_3)$
with $w=[4^{-3},1^{-1},2^2,3^2]$.
The tableaux in the middle and right
correspond to $\psi(w,-2)$ and $\psi(w,-1)$ respectively.
Using the criterion for $m$-minimal alcoves, one can check that
$\psi(w,-2)$ is a 3-minimal alcove in $\aA_+(C_3)$
while $\psi(w,-1)$ is not.
% More directly, checking for instance
%the Shi conditions on  the (circled) entry $k_{26}=4$, one can see that
%$H_{\alpha_{14},4}$ is a separating wall of the alcove, and thus $\psi(w,-1)$
%is not  3-minimal.
It is therefore natural to inquire  which integers $k$   preserve
$m$-minimality between the alcoves corresponding to $w$ and $\alpha(\w{k})$.
Since this is one of the main goals of this parer, we formulate it as a
question.

\begin{question}
  If $w\aA_0$ in an $m$-minimal alcove in $\aA_+(A_{n-1})$,
  for which choices of $k\in\mathbb Z$ does
   $\alpha(\w{k})$  correspond to an  $m$-minimal
  alcove  in $\aA_+(C_n)$?
\end{question}
The answer is given in the next proposition.
\begin{proposition}
\label{prop:conditions}
Let $w\alc_0\in \alc_+(A_{n-1})$ be an $m$-minimal alcove  with
	$\vec{n}(w)=(\beta_1,\dots,\beta_{n})$. 
  If we write each integer $k$  as $k=r+n\ell$, where $r\in\{0,1,\dots,n-1\}$ 
  and 	$\ell\in\mathbb{Z}$, then $\alpha(\w{k})$ corresponds to  an 
  $m$-minimal
	alcove in $\aA_+(C_n)$ if and only if 	one of the following 	conditions
	holds:
	\begin{enumerate}[\rm(i)]
	 \item $r=0$ and $-\lfloor\frac{m}{2}\rfloor-\beta_n\leq\ell\leq
	 \lfloor\frac{m+1}{2}\rfloor-\beta_1$, or
	 \item $r<n$ and
	 $-\beta_{n-r}-\lfloor\frac{m}{2}\rfloor\leq\ell\leq-\beta_{n-r+1}+\lfloor\frac{m-1}{2}\rfloor$.

	\end{enumerate}
\end{proposition}

\begin{proof}
(i)
	Suppose first that $r=0$.
	Then the $k$-shift of $w$ has level vector
  $$
	\vec{n}(\w{k})=(\beta_1+\ell,\dots,\beta_n+\ell,
      -\beta_n-\ell,\ldots,-\beta_1-\ell).
      $$
Recalling that the level vector $\vec{n}$ of $\alpha(\w{k})$ is the
antisymmetric expansion of $\vec{n}(\w{k})$, we claim
that it suffices to apply the criterion of Theorem~\ref{thm:mconditions} for
the following two pairs:\\
%	\item
\emph{First and last entry of $\vec{n}$:\;}  $(2n)^{-\beta_1-\ell}$ and
$1^{\beta_1+\ell}$,
\\
\emph{Two middle entries of $\vec{n}$:\;} $n^{\beta_n+\ell}$ and
$(n+1)^{-\beta_n-\ell}$. \\
%	\end{enumerate}
%   The other pairs of consecutive integers already satisfy the criteria
%  of Theorem \ref{thm:mconditions},
   Indeed, since all  other pairs of consecutive integers in $\vec{n}(\w{k})$
     are  shifted by $\ell$, the differences to be checked
     are the same as those for $\vec{n}(w)$.
     These, however, satisfy Theorem \ref{thm:mconditions}
     by the fact that $w\aA_0$ is $m$-minimal.
     Same applies for the upper half of $\vec{n}$ as well.

  For the first pair we require that  $\beta_1+\ell-(-\beta_1-\ell)-1\leq{m}$
	or equivalently $\ell\leq\frac{m+1}{2}-\beta_1$,
	while for the second we require
	$-\beta_n-\ell-(\beta_n+\ell)\leq{m}$
	or equivalently $\ell\geq -\frac{m}{2}-\beta_n$.
	Combining the above two inequalities and bearing in mind that $\ell$ is an
	integer, we arrive at the following range for $\ell$
	\[-\lfloor\tfrac{m}{2}\rfloor-\beta_n\leq\ell\leq\lfloor\tfrac{m+1}{2}\rfloor-\beta_1.\]
%%%%%%%%%%%%%%%%%%
(ii)
	Suppose now that $r<n$. Then the lower half of the level vector
   $\vec{n}$ of $\alpha(\w{k})$
%  $\vec{n}(\w{k})$
   becomes
	\[(\beta_{n-r+1}+\ell+1,\dots,
	\beta_n+\ell+1,\beta_1+\ell,\beta_2+\ell,\dots,\beta_{n-r}+\ell).\]
	Arguing as before, it suffices to apply the criterion of
	Theorem~\ref{thm:mconditions} for the
	following two pairs: \\
\emph{First and last entry of $\vec{n}$:\;} $(2n)^{-\beta_{n-r+1}-\ell-1}$ and
$1^{\beta_{n-r+1}+\ell+1}$, and\\
\emph{Two middle entries of $\vec{n}$:\;} $n^{\beta_{n-r}+\ell}$ and
$(n+1)^{-\beta_{n-r}-\ell}$. \\
	For the first pair we require that
	$\beta_{n-r+1}+\ell+1-(-\beta_{n-r+1}-\ell-1)-1\leq{m}$
	or equivalently $\ell\leq-\beta_{n-r+1}+\frac{m-1}{2}$,
	while for the second we require 
	$-\beta_{n-r}-\ell-(\beta_{n-r}+\ell)\leq{m}$
	or equivalently $\ell\geq-\beta_{n-r}-\frac{m}{2}$.
	Combining the two inequalities, we deduce that the range of
  $\ell$ in this case is
	\[-\beta_{n-r}-\lfloor\tfrac{m}{2}\rfloor\leq\ell\leq
	-\beta_{n-r+1}+\lfloor\tfrac{m-1}{2}\rfloor.\qedhere\]
\end{proof}
\medskip
For each $m$-minimal alcove $w\aA_0$ in $\alc_+(A_{n-1})$,
we call the set
 $\mgk{w}{m}$  of integers $k$ defined in Proposition
\ref{prop:conditions},
the \emph{$m$-admissible} set of $w$.
Although the numbers in $\mgk{m}{w}$ seem to be rather random,
they hold the answer to our bijection.
They are as many as the should;
each $m$-admissible set has $mn+1$ elements.
We prove this in the next proposition.

%
%Counting the possible values of $k$
%that satisfy the conditions of Lemma~\ref{prop:conditions}, we obtain
%the following.
%Write more here.............
%%\et{}{Shall we rename it as a Proposition?}

\begin{proposition}
\label{prop:counting}
	For each $m$-minimal alcove $w\alc_0$ in $\aA_+(A_{n-1})$
	there exist exactly $mn+1$ distinct integer values  $k\in\mathbb{Z}$
	such that $\alpha(\w{k})$ corresponds to an $m$-minimal alcove in
	$\aA_+(C_n)$.
\end{proposition}

\begin{proof}
  For $0\leq{r}\leq{n-1}$, let $c_r$ be the number of all possible
  integers $k$
  of the form 	$r+n\ell$, $\ell\in\mathbb{Z}$, for which $\alpha(\w{k})$
  corresponds to an 	$m$-minimal alcove in $\aA_+(C_n)$.
	Proposition~\ref{prop:conditions} implies that
	\begin{equation}
		c_r=\begin{cases}m+\beta_n-\beta_1+1, &\mbox{ if }\;\;\; r=0,\\
		\vspace{-4 mm} \\
		m+\beta_{n-r}-\beta_{n-r+1},&\mbox{ if }\;\;\;r\leq{n-1}.
		\end{cases}
	\end{equation}
	Summing over all $r$ we have
	\begin{align*}
	\sum\limits_{r=0}\limits^{n-1}c_r
	 & = m+\beta_n-\beta_1+1+
	\sum\limits_{r=1}\limits^{n-1}(m+\beta_{n-r}-\beta_{n-r+1})\\
	  & =
	m+\beta_n-\beta_1+1+(n-1)m+
	\sum\limits_{r=1}\limits^{n-1}(\beta_{n-r}-\beta_{n-r+1})=mn+1,
	 \end{align*}
	which completes our proof.
\end{proof}
Let $\mathscr{I}_m: \mgk{w}{m}=\{k_1<\cdots<k_{mn+1}\}\mapsto \{1,\ldots,mn+1\}$
be the bijection which sends  each  $k_i\in\mgk{w}{m}$ to its index, when  the
elements of  $\mgk{w}{m}$ are in increasing order.
Propositions \ref{prop:proj alc} and \ref{prop:conditions} imply that, if we
restrict the second argument of $\psi$ to
$\mgk{m}{w}\subseteq\mathbb{Z}$, $\psi$  maps each
$m$-minimal alcove in $\aA_+(A_{n-1})$ to an $m$-minimal alcove in
$\aA_+(C_n)$. Identifying the set $\mgk{w}{m}$ with that of
its indices and bearing in mind  that $m$-minimal alcoves in $\aA_+(A_{n-1})$ 
and $\aA_+(C_{n})$ are in bijection with dominant regions in \mshi{A_{n-1}} and
\mshi{C_{n}} respectively,  and  we arrive to our main theorem.
\medskip
\begin{theorem}
 \label{theor:main}
 Let $(\rR,i)\in \rR_+^m(A_{n-1})\times\{1,\ldots,mn+1\}$.
 Let $w\aA_0$ be the $m$-minimal alcove of the region $\rR$ and
 $k_i$ be the $i$-th element of the $m$-admissible set $\mgk{m}{w}$
 of $w$.
The map \bj which sends  each pair $(\rR,i)\in
\rR_+^m(A_{n-1})\times\{1,\ldots,mn+1\}$
to the region $\rR'\in \rR_+^m(C_{n})$ whose $m$-minimal alcove corresponds to
$\w{k_i}$, is a bijection.
\end{theorem}
%\smallskip
The map \bj\, of Theorem \ref{theor:main} is a combination  of previously
defined bijections.
Actually, if we ignore the first and last step that biject  each
region to its $m$-minimal alcove, the forward direction of \bj\ can be
stated more clearly: consider an $m$-minimal alcove $w\aA_0$ in
$\aA_+(A_{n-1})$,  compute  the $m$-admissible set
$\mgk{w}{m}=\{k_1,\ldots,k_{mn+1}\}$ and send each $(w,i)$ to
the antisymmetric expansion of its $k_i$-shift $\w{k_i}$.
%Proposition \ref{prop:conditions} certifies that the latter corresponds to an
%$m$-minimal alcove in $\aA_+(C_n)$.
\medskip

 In order to help the reader clarify the steps, we present
an explicit example  in Figure \ref{fig:step_by_step}. The first and last step
of the figure correspond to the way we associate each dominant region in
\mshi{\Phi} to its $m$-minimal alcove.
%, a procedure described in Appendix \ref{appA}.
More precisely, each region is encoded
by its {\em region Shi tableau}, which is a tableau defined in a way completely
analogous to that of an alcove. Since this correspondence is not indispensable
in our bijections, we describe it in Appendix \ref{appA}.
%\smallskip
\begin{figure}[h]
  \begin{center}
\begin{tikzpicture}[scale=0.85] 
\begin{scope}[scale=0.5]  
  
 \foreach \j in {0,...,6} {\foreach \i in {0,...,\j} { \draw (\i,\j)-- 
 (1+\i,\j); 
  \draw (1+\i,\j)-- (1+\i,1+\j); \draw (1+\i,1+\j)-- (\i,1+\j);
  \draw (\i,1+\j)-- (\i,\j);
   } }
   
  \node at (.5, 6.5){3};\node at(1.5,6.5){3};\node at(2.5,6.5){3};\node 
   at(3.5,6.5){3};\node at(4.5,6.5){2};\node at(5.5,6.5){1};\node 
   at(6.5,6.5){1};
   \node at (.5, 5.5) {3}; \node at (1.5, 5.5) {3}; \node at (2.5, 5.5) {3}; 
   \node at (3.5,5.5) {3};\node at (4.5,5.5) {1};\node at (5.5,5.5) {0}; 
   \node at (.5, 4.5) {3}; \node at (1.5, 4.5) {3}; \node at (2.5, 4.5) {3}; 
   \node at (3.5,4.5) {2}; \node at (4.5,4.5)  {0}; 
   \node at (.5, 3.5) {3}; \node at (1.5, 3.5) {3}; \node at (2.5, 3.5) {2}; 
   \node at (3.5, 3.5){1};
   \node at (.5, 2.5) {2}; \node at (1.5, 2.5) {1}; \node at (2.5, 2.5) {0}; 
   \node at (.5, 1.5) {1}; \node at (1.5, 1.5) {0}; 
   \node at (.5,.5)   {1};
\node at (1,-1){\begin{tabular}{l}\fs tableau of the region\\ 
  \fs $\rR'\in\rR_+^3(C_4)$ \end{tabular}};
%  ,path fading=west
% first 
  \draw[xshift=6.3cm,yshift=3.2cm,mycyan2!70,line width=3pt, 
  -stealth,  rotate=0] 
 (0,0).. controls (1,0.2) and (2,0.2)  .. (3,0);
\node at (7.8,4.1){\color{mycyan2}\fs Corollary A.2};
\end{scope}  

\begin{scope}[scale=0.5,xshift=11cm]  
\foreach \j in {0,...,6} {\foreach \i in {0,...,\j} { \draw (\i,\j)-- 
(1+\i,\j); 
    \draw (1+\i,\j)-- (1+\i,1+\j); \draw (1+\i,1+\j)-- (\i,1+\j);
    \draw (\i,1+\j)-- (\i,\j);
  } }
  
  \node at (.5, 6.5){6};\node at(1.5,6.5){5};\node at(2.5,6.5){4};\node 
  at(3.5,6.5){4};\node at(4.5,6.5){2};\node at(5.5,6.5){1};\node at(6.5,6.5){1};
  \node at (.5, 5.5) {5}; \node at (1.5, 5.5) {4}; \node at (2.5, 5.5) {3}; 
  \node at (3.5,5.5) {3};\node at (4.5,5.5) {1};\node at (5.5,5.5) {0}; 
  \node at (.5, 4.5) {4}; \node at (1.5, 4.5) {3}; \node at (2.5, 4.5) {3}; 
  \node at (3.5,4.5) {2}; \node at (4.5,4.5)  {0}; 
  \node at (.5, 3.5) {4}; \node at (1.5, 3.5) {3}; \node at (2.5, 3.5) {2}; 
  \node at (3.5, 3.5){1}; \node at (.5, 2.5) {2}; \node at (1.5, 2.5) {1}; 
  \node at (2.5, 2.5) {0}; 
  \node at (.5, 1.5) {1}; \node at (1.5, 1.5) {0}; 
  \node at (.5,.5)   {1};
%%%  text 
\node at (2,-1){\begin{tabular}{l}\fs tableau of the\\ 
 \fs 3-minimal alcove of $\rR'$ \end{tabular}};

%% second 
 \draw[xshift=8cm,yshift=3cm,mycyan2!70,line width=3pt, 
 -stealth,rotate=-30] 
 (0,0).. controls (1,0.5) and (2,0.5)  .. (3,0);
 \node at (11.5,3.5){\color{mycyan2}\begin{tabular}{r}\fs Proposition 2.2 \\
      \fs Lemma 2.3 \end{tabular}};

  \node at (15,1){$\bar w = 
  [2^{-3},3^{-2},8^{-2},5^{-1},4^1,1^2,6^2,7^3]$};  

%%%%%%%%%%%%%%%%%%%%%%%%%%%%%%%%%%
 \draw[xshift=11cm,yshift=0cm,mycyan2!70,line width=3pt, 
 -stealth,rotate=-85] 
 (0,0).. controls (1,0.2) and (2,0.2)  .. (3,-0.2);
%%%%%%%%%%%%%%%%%%%%%%%%%%%%% 
  \node at (15,-1.5){\color{mycyan2}
    \begin{tabular}{c}\fs Propostition 2.11\\
                       \fs $\bar w \mapsto (w,k)$
                       \end{tabular}}; 
%%%%%%%%%%%%%%%%%%%%%%%%%%%%%%                      
\node at (15,-5){\begin{tabular}{l}
  $w_{[k]}=[2^{-3},3^{-2},4^1,1^2]$\\
  $k= -3-2+1+2= \color{black!50!mycyan!80}{-2}$ \\ 
  $w=  w_{[k]}-k =[4^{-3},1^{-1},2^2,3^2]$
  \end{tabular}};
\end{scope}

\begin{scope}[scale=0.5,xshift=10cm,yshift=-10]

 \draw[xshift=10cm,yshift=-7cm,mycyan2!70,line width=3pt, 
 -stealth,rotate=-150] 
 (0,0).. controls (1,0.5) and (2,0.5)  .. (3,0);

\node at (15.5,-10){\color{mycyan2}
  \begin{tabular}{l}
 \fs  Proposition 2.1: 
  $w\mapsto {\mathsf T}(w)$ \\
   \fs  Proposition 2.12: \\
   \fs $\mathscr K_{3}(w)=\left\{-11,-10-7,-6,-3,\right.$
    \\ \fs \hspace{1.3cm}$ 
   \underset{\scalebox{1.2}{\color{black!50!mycyan!80}\!\!\!$k_6$}}{
     {\color{black!50!mycyan!80}\circled{-2}}},1,5,8,9,12,13,17$
   $\!\!\!\left.\right\}$
  \end{tabular}
  };

\end{scope}

\begin{scope}[scale=0.5,xshift=11cm,yshift=-10cm]

\foreach \j in {0,...,2} {\foreach \i in {0,...,\j} { \draw (\i,\j)-- 
(1+\i,\j); 
    \draw (1+\i,\j)-- (1+\i,1+\j); \draw (1+\i,1+\j)-- (\i,1+\j);
    \draw (\i,1+\j)-- (\i,\j);
  } }
  \node at (.5, 2.5) {4}; \node at (1.5, 2.5) {4}; \node at (2.5, 2.5) {1}; 
  \node at (.5, 1.5) {3}; \node at (1.5, 1.5) {3}; 
  \node at (.5,.5)   {0};
  
  \node at (4,1){$i=6$};
  
  \node at (3,-2){
    \color{mycyan2}\begin{tabular}{l}
   \fs  ${\sf T}(w)$ is the tableau of a
    \\ \fs $3$-minimal alcove of a \\
    \fs 
    region  $\mathcal R\in {\mathcal R}_+^3(A_3)$
    \end{tabular}};
  
  \draw[xshift=-1.5cm,yshift=1cm,mycyan2!70,line width=3pt, 
  -stealth,rotate=-180] 
  (0,0).. controls (1,0.2) and (2,0.2)  .. (3,0);
   \node at (-3.5,0){\color{mycyan2}\fs Corollary A.2};
  \end{scope}

\begin{scope}[scale=0.5,xshift=0cm,yshift=-10cm]
\foreach \j in {0,...,2} {\foreach \i in {0,...,\j} { \draw (\i,\j)-- 
(1+\i,\j); 
    \draw (1+\i,\j)-- (1+\i,1+\j); \draw (1+\i,1+\j)-- (\i,1+\j);
    \draw (\i,1+\j)-- (\i,\j);
  } }
  \node at (.5, 2.5) {3}; \node at (1.5, 2.5) {3}; \node at (2.5, 2.5) {1}; 
  \node at (.5, 1.5) {3}; \node at (1.5, 1.5) {3}; 
  \node at (.5,.5)   {0};
  
    \node at (3.5,1){$i=6$};
    
     \node at (2,-1.4){
       \color{mycyan2}       
       \begin{tabular}{l}
       \fs tableau of the region \\
       \fs $\mathcal R \in {\mathcal R}_+^3(A_3)$
       \end{tabular}       
       };
\end{scope}
  \end{tikzpicture}
\end{center}
\caption{
    The arrows indicate the steps for the inverse of the map
    \bj, each with  reference  to the required proposition/lemma.}
  \label{fig:step_by_step}  
\end{figure}
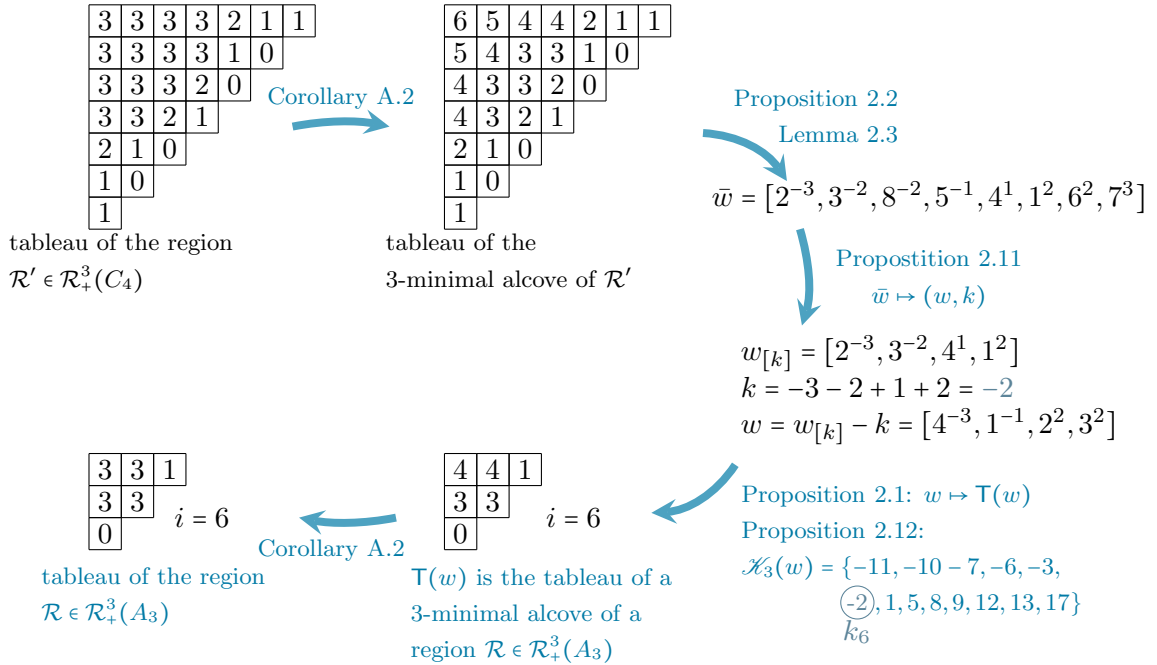  

%\begin{figure}[h]
%  \begin{center}
%%\includegraphics[scale=0.5]{fig:step_by_step}
%\includegraphics[scale=0.5]{figure2}
%\caption{
%  The compact arrows indicate the steps for the inverse of the map
%  \bj,  while  the  dashed arrows indicate its forward direction (each step has
%  reference
%  to the required proposition/lemma).}
%\label{fig:step_by_step}
%\end{center}
%\end{figure}

Theorem \ref{theor:main}  implies that ignoring the second argument of
the map  $\mbox{\bj\!\!}^{-1}$ we get a surjection.
\begin{corollary}
\label{cor:dregC}
The map $\phi:\mathcal{R}_+^m(C_n)\to\mathcal{R}_+^m(A_{n-1})$ defined
by  $\phi(\rR)=\pr{1}({\sf Bj}_1^{-1}(\rR))$ is a surjection.
\end{corollary}

The map $\phi$ partitions  $\rR_+^m(C_n)$ into sets of regions each having the
same image under $\phi$. Is there a way to geometrically understand this
surjection? Since regions of $\rR_+^m(C_{n})$ are viewed as regions in
$\rR_+^m(A_{2n-1})$, is it reasonable to ask if $\phi$ corresponds to some
affine projection from $\mathbb R^{2n}$ to $\mathbb R^n$.
In Figure \ref{fig:proj} we illustrate a simple example of this surjection.

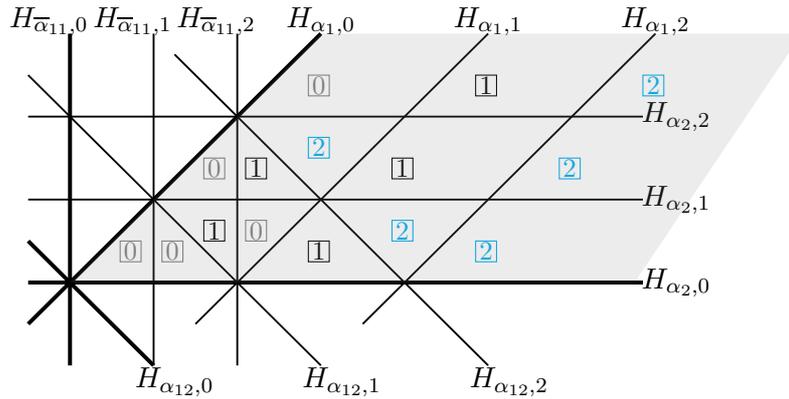
\begin{figure}[h]
 \begin{center}
\begin{tikzpicture}[scale=0.55]
\draw[line width= 1.5pt] (-1,2)--(13.7,2);
\draw[line width= 0.7pt] (-1,4)--(13.7,4);
\draw[line width= 0.7pt] (-1,6)--(13.7,6);
%\draw[line width= 0.7pt] (-1,5)--(11,5);

\draw[line width= 1.5pt] (2,0)--(-1,3);
\draw[line width= 0.7pt] (6,0)--(-1,7);
\draw[line width= 0.7pt] (10,0)--(2.5,7.5);
%\draw[line width= 0.7pt] (8,0)--(-1,9);

\draw[line width= 1.5pt] (0,0)--(0,8);
\draw[line width= 0.7pt] (2,0)--(2,8);
\draw[line width= 0.7pt] (4,0)--(4,8);
%\draw[line width= 0.7pt] (3,0)--(3,8);

\draw[line width= 1.5pt] (-1,1)--(6,8);
\draw[line width= 0.7pt] (3,1)--(10,8);
\draw[line width= 0.7pt] (7,1)--(14,8);
%\draw[line width= 0.7pt] (5,1)--(12,8);

\node at (14.5,2)  {$H_{\alpha_2,0}$};
\node at (14.5,4)  {$H_{\alpha_2,1}$};
\node at (14.5,6)  {$H_{\alpha_2,2}$};
%\node at (11.5,5)  {$H_{\alpha_2,3}$};

\node at (6,8.3) {$H_{\alpha_{1},0}$};
\node at (10,8.3) {$H_{\alpha_{1},1}$};
\node at (14,8.3) {$H_{\alpha_{1},2}$};
%\node at (12.5,8.5) {\rotatebox{45}{$H_{\alpha_{1},3}$}};

\node at (2.5,-0.4)   {$H_{\alpha_{12},0}$};
\node at (6.5,-0.4) {$H_{\alpha_{12},1}$};
\node at (10.5,-0.4)  {$H_{\alpha_{12},2}$};
%\node at (8.5,-0.6) {\rotatebox{-45}{$H_{\alpha_{12},3}$}};

\node at (-0.5,8.3){$H_{\overline\alpha_{11},0}$};
\node at (1.5,8.3)  {$H_{\overline\alpha_{11},1}$};
\node at (3.5,8.3)  {$H_{\overline\alpha_{11},2}$};
%\node at (3.2,8.4)  {$H_{\overline\alpha_{11},3}$};

\begin{scope}[scale=0.5, xshift=2.4 cm, yshift=5 cm, gray]
\draw (0,0)rectangle(1,1);
\node at(0.5,0.5){\small 0};
\end{scope}
\begin{scope}[scale=0.5, xshift=4.4 cm, yshift=5 cm, gray]
\draw (0,0)rectangle(1,1);
\node at(0.5,0.5){\small 0};
\end{scope}
\begin{scope}[scale=0.5, xshift=6.4 cm, yshift=6 cm]
\draw (0,0)rectangle(1,1);
\node at(0.5,0.5){\small 1};
\end{scope}
\begin{scope}[scale=0.5, xshift=8.4 cm, yshift=6 cm, gray]
\draw (0,0)rectangle(1,1);
\node at(0.5,0.5){\small 0};
\end{scope}
\begin{scope}[scale=0.5, xshift=11.4 cm, yshift=5 cm]
\draw (0,0)rectangle(1,1);
\node at(0.5,0.5){\small 1};
\end{scope}
\begin{scope}[scale=0.5, xshift=15.4 cm, yshift=6 cm,cyan]
\draw (0,0)rectangle(1,1);
\node at(0.5,0.5){\small 2};
\end{scope}
\begin{scope}[scale=0.5, xshift=19.4 cm, yshift=5 cm,cyan]
\draw (0,0)rectangle(1,1);
\node at(0.5,0.5){\small 2};
\end{scope}
\begin{scope}[scale=0.5, xshift=6.4 cm, yshift=9 cm, gray]
\draw (0,0)rectangle(1,1);
\node at(0.5,0.5){\small 0};
\end{scope}\begin{scope}[scale=0.5, xshift=8.4 cm, yshift=9 cm]
\draw (0,0)rectangle(1,1);
\node at(0.5,0.5){\small 1};
\end{scope}\begin{scope}[scale=0.5, xshift=11.4 cm, yshift=10 cm, cyan]
\draw (0,0)rectangle(1,1);
\node at(0.5,0.5){\small 2};
\end{scope}\begin{scope}[scale=0.5, xshift=15.4 cm, yshift=9 cm]
\draw (0,0)rectangle(1,1);
\node at(0.5,0.5){\small 1};
\end{scope}\begin{scope}[scale=0.5, xshift=23.4 cm, yshift=9 cm, cyan]
\draw (0,0) rectangle (1,1);
\node at(0.5,0.5){\small 2};
\end{scope}
\begin{scope}[scale=0.5, xshift=11.4 cm, yshift=13 cm, gray]
\draw (0,0)rectangle(1,1);
\node at(0.5,0.5){\small 0};
\end{scope}
\begin{scope}[scale=0.5, xshift=19.4 cm, yshift=13 cm]
\draw (0,0)rectangle(1,1);
\node at(0.5,0.5){\small 1};
\end{scope}
\begin{scope}[scale=0.5, xshift=27.4 cm, yshift=13 cm, cyan]
\draw (0,0) rectangle (1,1);
\node at(0.5,0.5){\small 2};
\end{scope}
\fill[color=gray,opacity=0.15]
(0,2)--(13.5,2)--(17.5,8)--(6,8);
\end{tikzpicture}
\end{center}
\caption{ The arrangement $\mnshi{2}{C_2}$, depicted above, has 15 dominant 
regions. The arrangement $\mnshi{2}{A_1}$ has 3 dominant regions,
 with  Shi tableau  $\boxed{0},\boxed{1}$ and $\boxed{2}$. In view of Corollary 
 \ref{cor:dregC}, the regions in $\mnshi{2}{C_2}$ are partitioned into three
equinumerous sets, having the same image under $\phi$.}
\label{fig:proj}
\end{figure}

%\subsection{Type B discussion }
%In this brief paragraph, 
\medskip
We conclude this section by  discussing the problems we encounter when 
we try to formulate an analogue of Theorem \ref{theor:main} for the type $B$ 
case. Without delving into details, we mention that  Shi tableaux of type $B_n$ 
coincide with self-conjugate Shi tableaux of type $A_{2n}$
whose main diagonal is empty \cite{sh-nost-97}.
Since the map $\mathsf T$ in Proposition \ref{prop:inverse}  cannot be 
applied unless all Shi conditions are known, 
we seek a unique way to determine the empty entries
of a type $B_n$ Shi tableau, so that Shi conditions are preserved. 
With this in mind, there is a natural way to go from the 
Shi tableau $T'$ of an alcove $w'\aA_0\in\aA_+(C_n)$
to  the tableau $T$ of an alcove $w\aA_0\in\aA_+(B_n)$: 
if $\delta_1,\ldots,\delta_n$ are the entries of the main diagonal of $T'$,
insert an $n$-th column and an $n$-th row to $T'$  with entries 
$\lfloor\tfrac{\delta_i-1}{2}\rfloor$ and then delete 
all $\delta_i$'s from  the main diagonal (see Figure \ref{typeBfig}). 
 It is not hard to see that the inserted entries
 preserve the Shi conditions on the new tableau $T$. 
In terms of  abacus diagrams, the above procedure corresponds to 
inserting  $n^0$ in  the (central entry of the) diagram of $w'$  and
replacing each $k^{\ell_k}$ with $k>n$ by $(k+1)^{\ell_k}$. 
Two  examples of this map are shown in Figure \ref{typeBfig}.

Unfortunately, the map described above in not a bijection. 
As one can see in Figure \ref{typeBfig}, two different 
tableaux $T'$ of alcoves in $\aA_+(C_3)$
map to the same tableau in $\aA_+(B_3)$.
Even our hope of it being a bijection when restricted to 
$m$-minimal alcoves is dissolved by the same example;
both type $C$ tableau, which are $3$-minimal,  map 
on the same type B tableau. 
Thus, our wish to associate each $m$-minimal alcove in $\aA_+^m(C_n)$
with one such alcove in $\aA_+^m(B_n)$, cannot be worked out.

Maybe, the only approach to find the desired bijection, is 
to use the formal definition of $\widetilde B_n$ as the subgroup 
of  even permutations of $\widetilde C_{n}$ \cite[Section 2]{ahj-rcosc-13},
\cite[Section 8.5]{bb_ccg_04}. In this case though, it is not evident how 
one could exploit evenness to produce self-conjugate Shi tableau
with empty main diagonal. 

\begin{figure}[h]
\begin{tikzpicture}[scale=0.65]
\begin{scope}[scale=0.8,xshift=0cm,yshift=0cm] 
\draw[line width= 0.5pt](1,1)rectangle(2,4);  
\draw[line width=0.5pt](2,2)rectangle(3,4);
\draw[line width= 0.5pt](3,3)rectangle(4,4);
\draw[line width= 0.5pt](1,2)--(2,2);
\draw[line width= 0.5pt](1,3)--(3,3);
\node at(1.5,3.5) {$3$};\node at(2.5,3.5) {$2$};\node at(3.5, 3.5){$0$};
\node at(1.5,2.5) {$2$};\node at(2.5,2.5) {$1$};
\node at(1.5,1.5) {$0$};
\node at(5,2){$\Longrightarrow$}; 

\node[text width=7cm]at(6,-1){\footnotesize
  $w'=[4^{-2},3^{-1},2^1,1^2]$ \hspace{0.4cm} $w=[5^{-2},4^{-1},3^0,2^1,1^2]$
%  
%  The $\xcancel{3}$ on the top left of the type $B$ tableau, 
%  implies an inequality that cannot be deduced from the others. 
%  This should not happen. 
};
\end{scope}
\begin{scope}[scale=0.8,xshift=6cm,yshift=0cm] 
\draw[line width= 0.5pt](0,0)rectangle(1,4);
\draw[line width= 0.5pt](1,1)rectangle(2,4);  
\draw[line width=0.5pt](2,2)rectangle(3,4);
\draw[line width= 0.5pt](3,3)rectangle(4,4);
\draw[line width= 0.5pt](0,1)--(1,1);
\draw[line width= 0.5pt] (0,2)--(2,2);
\draw[line width= 0.5pt] (0,3)--(3,3);
\node at(.5,3.5){$\xcancel{3}$};\node at(1.5,3.5){$2$};\node at(2.5,3.5) 
{$\color{gray}1$};\node at(3.5, 3.5){$0$};
\node at(.5,2.5){$2$};\node at(1.5,2.5){$\xcancel{1}$};
\node at(2.5,2.5){$\color{gray}0$};
\node at(.5,1.5){$\color{gray}1$};\node at(1.5,1.5){$\color{gray}0$};
\node at(.5,.5){$0$}; 
\end{scope}
%%%%%%%%%%%%%%%%%%%%%%%%%%%%%%%%%%%%%%%%%%%%%%%%%%%%%%%%%%%%%%%%%
%%%%%%%%%%%%%%%%%%%%%%%%%%%%%%%%%%%%%%%%%%%%%%%%%%%%%%%%%%%%%%%%%

\begin{scope}[scale=0.8,xshift=15cm,yshift=0cm] 
\draw[line width= 0.5pt](1,1)rectangle(2,4);  
\draw[line width=0.5pt](2,2)rectangle(3,4);
\draw[line width= 0.5pt](3,3)rectangle(4,4);
\draw[line width= 0.5pt] (1,2)--(2,2);
\draw[line width= 0.5pt] (1,3)--(3,3);
\node at(1.5,3.5) {$2$};\node at(2.5,3.5) {$2$};\node at(3.5, 3.5){$0$};
\node at(1.5,2.5) {$2$};\node at(2.5,2.5) {$1$};
\node at(1.5,1.5) {$0$};
\node at(5,2){$\Longrightarrow$}; 

\node[text width=7cm]at(6,-1){\footnotesize 
  $w'=[1^{-1},3^{-1},2^1,4^1]$ \hspace{0.4cm} $w=[1^{-1},4^{-1},3^0,2^1,5^1]$
%  
%  This is not happening is this case. 
%  I guess this should be accepted as type $B$ 3-minimal.   
}; 
\end{scope}

\begin{scope}[scale=0.8,xshift=21cm,yshift=0cm] 
\draw[line width= 0.5pt](0,0)rectangle(1,4);
\draw[line width= 0.5pt](1,1)rectangle(2,4);  
\draw[line width=0.5pt](2,2)rectangle(3,4);
\draw[line width= 0.5pt](3,3)rectangle(4,4);
\draw[line width= 0.5pt](0,1)--(1,1);
\draw[line width= 0.5pt] (0,2)--(2,2);
\draw[line width= 0.5pt] (0,3)--(3,3);
\node at(.5,3.5){$\xcancel{2}$};\node at(1.5,3.5){$2$};\node at(2.5,3.5) 
{$\color{gray}1$};\node at(3.5, 3.5){$0$};
\node at(.5,2.5){$2$};\node at(1.5,2.5){$\xcancel{1}$};
\node at(2.5,2.5){$\color{gray}0$};
\node at(.5,1.5){$\color{gray}1$};\node at(1.5,1.5){$\color{gray}0$};
\node at(.5,.5){$0$}; 
\end{scope}
\end{tikzpicture}
\caption{}
\label{typeBfig}
\end{figure}

\section{Lattice paths and \bjj}
\label{sec:latt}
Given $n,m\geq{1}$, we denote by $\lp{m}{n}$ the set of all $N-E$ lattice paths 
from $(0,0)$ to $(n,mn)$ and by $\mathcal{D}_{n}^m$ the set of all  $m$-Dyck 
paths of height $n$. As we already mentioned in the introduction, the above 
sets are enumerated by Catalan numbers i.e.,  
$|\mathcal{D}_{n}^m|=\mcatn{m}{A_{n-1}}$
and  $|\lp{m}{n}|=\tbinom{mn+n}{n}=\mcatn{m}{C_n}$. Thus, the relation
$|\lp{m}{n}|=|\mathcal{D}_{n}^m|(mn+1)$ can be viewed as another instance of
\eqref{catAC} to be explained  bijectively. In Theorem  \ref{bij:paths} we
provide a natural, but rather hidden, bijection between the sets $\lp{m}{n}$
and $\dyp{m}{n}\times\{0,\dots,mn\}$.

Before continuing, we  introduce definitions and notation. Every path
$\pa\in\lp{m}{n}$ can uniquely be
determined by its step sequence $(s_1,\dots,s_n)$, where
$s_i$ denotes the number of east steps occurring before the $i$-th north step
(see Figure \ref{fig:shifted_partitions_1}).
Thus, we can write $\pa=(s_1,\dots,s_n)$ as well.
Notice that $0\leq s_1\leq \cdots\leq s_n\leq mn$
for all paths in $\lp{m}{n}$, while $\pa\in\dyp{m}{n}$
if and only if $s_i\leq m(i-1)$ for every $1\leq i\leq n$.

\begin{definition}
	For each  $\pa=(s_1,\dots,s_n)\in\lp{m}{n}$ and  $0\leq i\leq n-1$,
  we define its \em{$i$-th permutation} $d_i(\pa)$  to be the path in
  $\lp{m}{n}$
	with step sequence
	$(0,s_{i+2}-s_{i+1},\dots,s_n-s_{i+1},\bar{s}_1-s_{i+1},\bar{s}_2-s_{i+1},
  \dots,\bar{s}_{i}-s_{i+1})$,
	where $\bar{s}_j=s_j+mn+1$.
\end{definition}
Pictorially, one may think of $d_i(\pa)$ as follows.
For each $\pa\in\lp{m}{n}$, let $\pa''$ be the path
with step sequence $(s_1,s_2,\dots, s_n,\bar{s}_1,\bar{s}_2,\dots,\bar{s}_n)$
i.e., $\pa''$ is the concatenation of $\pa$, an east step $E$ and one more copy
of $\pa$. Then,  $d_i(\pa)$ corresponds to the subpath of $\pa''$ starting with 
its  $i$-th and finishing before its $(n+i)$-th north step.
Although in general $d_i(\pa)$ is not an $m$-Dyck path,
there exists a unique $i_0$ for which this is true.
This is the content of the next Proposition,
which is motivated from  \cite[Chapter 1.4]{mo-lpc-79}.

\begin{proposition}
	\label{prop:hiddendyck}
	For each $\pa\in\lp{m}{n}$ there exists a unique index $0\leq i_0\leq{n-1}$,
  for which $d_{i_0}(\pa)\in\dyp{m}{n}$.
\end{proposition}
\begin{proof}
  For each $\pa\in\lp{m}{n}$ let $\pa''$ be the path
  with step sequence
  $(s_1,s_2,\dots,s_n,\bar{s}_1,\bar{s}_2,\dots,\bar{s}_n)$.
  Set $O=(0,0)$,  $Q_1 = (mn,n)$, $Q_2 = (mn+1, n)$ and
  $R=(2mn+1,2n)$, so that  $\pa$ is the subpath of $\pa''$ from $O$ to $Q_1$
   and $\pa'$, its shifted by $(mn+1,n)$ copy, is the subpath from $Q_2$ to
   $R$. Next, among all lines of slope $\tfrac{1}{m}$,  consider the unique
  $\varepsilon_{k_0}:y=\tfrac{1}{m}x+k_0$ tangent-to  and containing  $\pa$
   in its closed upper half-plane $\overline H_{k_0}^+$. For every $0\leq
   i\leq n-1$ set $A_i=(s_i,i-1)$ and $\bar{A}_i=(\bar{s}_i,n+i-1)$. Notice
   that  $A_i\in\pa$, while  $\bar{A}_i$  is the corresponding point on $\pa'$.
  The line $\varepsilon_{k_0}$ has the property that  $\bar{A}_i\in\overline
  H_{k_0}^+$ if  and only if   $A_i\in H_{k_0}^+$.  Indeed,
  \begin{equation}
  \label{Ai}
  \begin{split}
  \bar{A}_i\in\overline H_{k_0}^+
  \Leftrightarrow\; & n+i-1\geq\tfrac{1}{m}\bar s_i+k_0\\
  \Leftrightarrow\; & n+i-1\geq\tfrac{1}{m}(s_i+mn+1)+k_0\\
  \Leftrightarrow\; & i-1\geq\tfrac{1}{m}(s_i+1)+k_0 >\tfrac{1}{m}s_i+k_0
  \Leftrightarrow\;  A_i\in H_{k_0}^+.
  \end{split}
  \end{equation}
  Let $i_0:=\min\{i\,|\,A_i \in \pa\cap \varepsilon_{k_0}\}$,
  so that $A_{i_0}$ is the first point of intersection of
  $\pa$ with $\varepsilon_{k_0}$.
  We claim that the subpath $\pa_{k_0}$ of $\pa''$ with endpoints
  $A_{i_0}$ and $B_{i_0}=\bar A_{i_0}-(1,0)$
  lies above the line $\varepsilon_{k_0}$,
  and thus defines an $m$-Dyck path of height $n$.
  Indeed, the part of $\pa_{k_0}$ from $A_{i_0}$
  to $Q_1$ is in $\overline H^+_{k_0}$ by
  construction of $\varepsilon_{k_0}$.
  The part of $\pa_{k_0}$ from  $Q_2$ to $B_{i_0}$ lies in
  $\overline H^+_{k_0}$
  since otherwise, in view of \eqref{Ai},  the existence of some  $\bar A_j$
  $\nin\overline H^+_{k_0}$  would imply the existence of an $A_j$ preceding
  $A_{i_0}$  with  $A_j\nin H^+_{k_0}$.  This would contradict the fact that
  $A_{i_0}$ is the first point where $\varepsilon_{k_0}$ intersects $\pa$.
  \end{proof}

 The reader is invited to follow the steps of the proof of Proposition
 \ref{prop:hiddendyck} in  Figure \ref{fig:shifted_partitions_1}.
 
 \begin{figure}[h]
\begin{center}
\begin{tikzpicture}[scale=0.35]
\begin{scope}
\fill[pattern=north west lines, pattern  color=mycyan!50] 
(0,2)--(11,2)--(11,7)--(0,7);

\fill[pattern=north east lines, pattern  color=mycyan2!50] 
(11,3)--(14,3)--(14,4)--(17,4)--(17,5)--(20,5)--(20,6)--(23,6)--(23,7)
--(11,7);

\draw [step=1.0,gray!50] (0,0) grid (15,5);

\draw [step=1.0,gray!50] (16,5) grid (31,10);

\draw (2,-1)--(35,10);
\node at (2,-1.4){\fs $\varepsilon_{k_0}$};

\node at (0,-0.6){\fs $P$};
\node at (0,0){\tiny $\bullet$};
\filldraw (0,0) circle (0.16cm);
\node at (11,1.3){\fs $A_{i_0}$};

\filldraw (15,5) circle (0.12cm);
\node at (15,5.6){\fs $Q_1$};
\filldraw (16,5) circle (0.12cm);
\node at (16.2,4.3){\fs $Q_2$};

\filldraw (31,10) circle (0.16cm);
\node at (31,10.5){\fs $R$};

\node at (26,6.3){\fs$B_{i_0}$};
\node at (27.9,6.4){\fs$\overline A_{i_0}$};
\draw[line width=1] 
(0,0)--(1,0)--(1,1)--(4,1)--(4,2)--(11,2)--(11,3)--(12,3)--(12,4)--(14,4)
--(14,5)--(15,5);

\draw[dotted,line width=1] (15,5)--(16,5); 

\draw[line width=1] (16,5)--(17,5)--(17,6)--(20,6)--(20,7)--(27,7)
--(27,8)--(28,8)--(28,9)--(30,9)--(30,10)--(31,10); 

\filldraw[fill=white] (27,7) circle (0.16cm);
\filldraw[fill=white] (11,2) circle (0.16cm);
\filldraw[fill=white] (12,3) circle (0.16cm);
\filldraw[fill=white] (14,4) circle (0.16cm);
\filldraw[fill=white] (17,5) circle (0.16cm);
\filldraw[fill=white] (20,6) circle (0.16cm);
\filldraw (26,7) circle (0.16cm);
\end{scope}

\end{tikzpicture}
\end{center}

 \caption{
   The path $\pa=ENEEENEEEEEEENENEENE$ from $P$ to $Q_1$
   has step sequence $(1,4,11,12,14)$ and belongs to $\lp{3}{5}$.
   The path $\pa''$ from $P$ to $R$, is the concatenation of $\pa$, $E$ and
   $\pa$. Since $m=3$, we consider the line $\varepsilon$  with slope
   $1/3$, tangent-to and containing $\pa$ in its upper half-space. Its
   first intersection point with $\pa$ is $A_{i_0}$. In the present
   case it is $A_3$. Thus, we deduce  that $i_0=3$ and
   $d_{i_0}(\pa)=(0,1,3,6,9)$.}
  \label{fig:shifted_partitions_1}
  
  \end{figure}
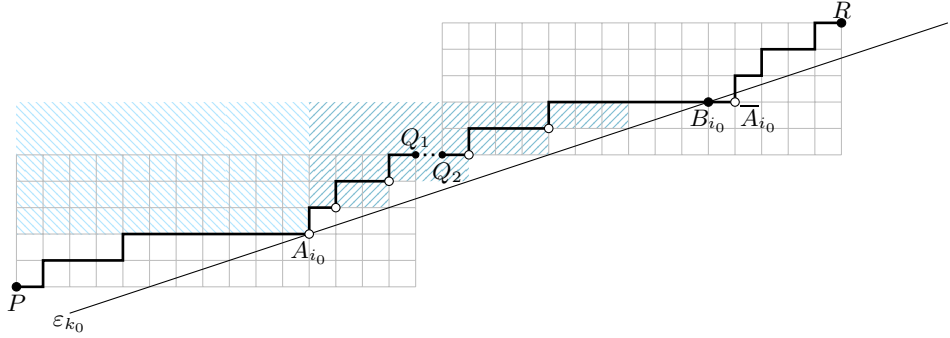
%  \begin{figure}[h]
%		\begin{center}
%%    \includegraphics[width=0.8\linewidth]{fig:shifted_partitions_1}
%    \includegraphics[width=0.8\linewidth]{figure3}
%    \caption{
%      The path $\pa=ENEEENEEEEEEENENEENE$ from $P$ to $Q_1$
%      has step sequence $(1,4,11,12,14)$ and belongs to $\lp{3}{5}$.
%       The path $\pa''$ from $P$ to $R$, is the concatenation of $\pa$, $E$ and
%       $\pa$. Since $m=3$, we consider the line $\varepsilon$  with slope
%       $1/3$, tangent-to and containing $\pa$ in its upper half-space. Its
%       first intersection point with $\pa$ is $A_{i_0}$. In the present
%       case it is $A_3$. Thus, we deduce  that $i_0=3$ and
%       $d_{i_0}(\pa)=(0,1,3,6,9)$.}
%    \label{fig:shifted_partitions_1}
%	\end{center}
%  \end{figure}

  The next theorem, which follows  from Proposition
  \ref{prop:hiddendyck}, describes (the inverse of) \bjj.
\begin{theorem}
\label{bij:paths}
The map $h:\lp{m}{n}\rightarrow \dyp{m}{n}\times\{0,\ldots,mn\}$
with $h(\pa)=(d_{i_0}(\pa),s_{i_0})$ is a bijection.
\end{theorem}
\begin{proof}
The forward direction of the map $h$ is immediate from Proposition
\ref{prop:hiddendyck}.
To prove the reverse, we consider an $m$-Dyck path $\pa=(s_1,\ldots,s_n)$
and an integer $0\leq{k}\leq{mn}$.
Let $1\leq{}j_0\leq{n}$ be the  unique index  for which $s_{j}+k\leq{mn}$
for all $j\leq{j_0}$ and $s_{j}>mn$ otherwise.
We set $$\overline\pa=(\bar s_1,\ldots,\bar s_n)
:=(s_{j_0+1}+k-mn-1,\ldots,s_n+k-mn-1,s_1+k,\ldots,s_{j_0}+k)$$
and show that $\overline\pa\in\lp{m}{n}$
(see Figure \ref{fig:shifted_partitions_reverse}
for a pictorial description of $\overline \pa$).
By the choice of $j_0$ and $k$ and recalling that
$s_i\leq m(n-i)$, one can easily deduce that $0\leq\bar s_i\leq{mn}$
for all $i$. Thus, for showing that $\pa\in\lp{m}{n}$,
 it remains to prove that $\bar s_1\leq\cdots\leq\bar s_n$. If
$j_0=n$ the claim is immediate (since then $\overline\pa=\pa$), while if
$j_0<n$  we only need to verify that
$s_{n}+k-mn-1\leq{}s_1+k$. Bearing in mind that $s_1=0$ for all Dyck paths, the
above simplifies to showing that $s_n\leq{mn+1}$, which is true for all
$\pa\in\dyp{m}{n}$.

We leave the reader to verify that the map sending each pair $(\pa,k)$
to $\overline\pa$ is indeed the reverse of $h$.
\end{proof}

\begin{figure}[h]
\begin{center}
\begin{tikzpicture}[scale=0.3]
\begin{scope}
\fill[pattern=north west lines, pattern  color=mycyan!50] 
(0,2)--(11,2)--(11,7)--(0,7);

\fill[pattern=north east lines, pattern  color=mycyan2!50] 
(11,3)--(14,3)--(14,4)--(17,4)--(17,5)--(20,5)--(20,6)--(23,6)--(23,7)
--(11,7);

\draw [step=1.0,gray!50] (0,2) grid (15,7);
\draw[line width=1,color=mycyan] (16,5)--(17,5)--(17,6)--(20,6)--(20,7)--(23,7);
\draw[color=mycyan,line width=1] (15,5)--(16,5); 
\filldraw (15,5) circle (0.12cm);
\node at (15,5.6){\fs $A$};
\filldraw (16,5) circle (0.12cm);
\node at (16,5.6){\fs $B$};

\draw[line width=1,dashed] (0,2)--(11,2);
\draw[line width=1] 
(11,2)--(11,3)--(12,3)--(12,4)--(14,4)--(14,5)--(15,5);
\filldraw[fill=white] (11,2) circle (0.16cm);
\filldraw[fill=white] (12,3) circle (0.16cm);
\filldraw[fill=white] (14,4) circle (0.16cm);
\filldraw[fill=white] (17,5) circle (0.16cm);
\filldraw[fill=white] (20,6) circle (0.16cm);

\node at (0,1.5){\fs$P$};
\node at (23.5,7){\fs$Q$};

\draw[xshift=24cm,yshift=4cm,mycyan!70,line width=3pt, 
-stealth,rotate=0] 
(0,0).. controls (1,0.2) and (2,0.2)  .. (3,0);

\end{scope}
%%%%%%%%%%%%%%%%%%%%%%%%%%%%%%%%%%%%%%%%%%%%

\begin{scope}[xshift=29cm]
\fill[pattern=north west lines, pattern  color=mycyan!50] 
(0,2)--(11,2)--(11,5)--(0,5);

\fill[pattern=north east lines, pattern  color=mycyan2!50] 
(11,3)--(14,3)--(14,4)--(17,4)--(17,5)--(20,5)--(11,5);

\draw [step=1.0,gray!50] (0,0) grid (15,5);

\node at (0,-0.6){\fs $B$};

\filldraw (15,5) circle (0.12cm);
\node at (15,5.6){\fs $A$};

\draw[line width=1,color=mycyan] 
(0,0)--(1,0)--(1,1)--(4,1)--(4,2)--(7,2);
\draw[line width=1,dashed]
(7,2)--(11,2); 

\draw[line width=1] 
(11,2)--(11,3)--(12,3)--(12,4)--(14,4)--(14,5)--(15,5);
\filldraw[fill=white] (11,2) circle (0.16cm);
\filldraw[fill=white] (12,3) circle (0.16cm);
\filldraw[fill=white] (14,4) circle (0.16cm);
\filldraw[color=mycyan,fill=white] (1,0) circle (0.16cm);
\filldraw[color=mycyan,fill=white] (4,1) circle (0.16cm);
%\filldraw[fill=white] (20,6) circle (0.16cm);
%\filldraw (26,7) circle (0.16cm);
\filldraw (0,0) circle (0.12cm);
\end{scope}

\end{tikzpicture}
\end{center}

\caption{Let $\pa=(0,1,3,6,9)\in\dyp{3}{5}$ and $k=11$. Consider the path $PQ$
  we obtain when shifting $\pa$ by 11 east steps, placed at the origin of    a
  $15\times 5$ grid (left). Let $AB$ be the first (horizontal) step of the path
  that exceeds the grid. The path $\overline \pa$ is obtained by switching the
  subpaths $PA$ and $BQ$, so that the latter begins at the origin and the former
  ends at the last point, of a $15\times{5}$ grid (right).
}
\label{fig:shifted_partitions_reverse}
  \end{figure}
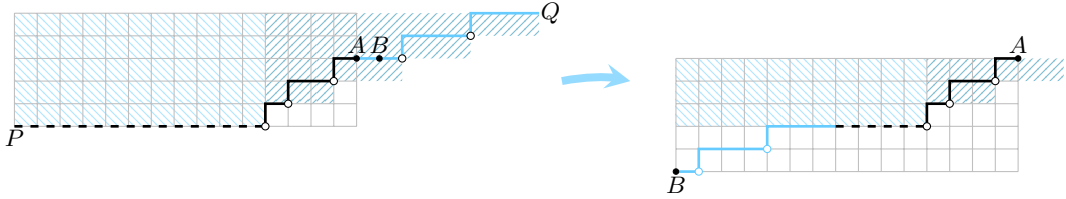

%\begin{figure}[h]
%  \centering
%%\includegraphics[width=0.85\linewidth]{fig:shifted_partitions_reverse}
%\includegraphics[width=0.85\linewidth]{figure4}
%\caption{Let $\pa=(0,1,3,6,9)\in\dyp{3}{5}$ and $k=11$. Consider the path $PQ$
%we obtain when shifting $\pa$ by 11 east steps, placed at the origin of    a
%$15\times 5$ grid (left). Let $AB$ be the first (horizontal) step of the path
%that exceeds the grid. The path $\overline \pa$ is obtained by switching the
%subpaths $PA$ and $BQ$, so that the latter begins at the origin and the former
%ends at the last point, of a $15\times{5}$ grid (right).
%  }
%\label{fig:shifted_partitions_reverse}
%\end{figure}

The following theorem completes the $\mbox{FKT}_1$ direction of the diagram of
Section \ref{intro}.
\begin{theorem}\cite[Theorem~3.1]{fkt-fgcc-13}
\label{bij:kij}
  There exists an explicit bijection between
  dominant regions in $\rR_+^m(A_{n-1})$
  and $m$-Dyck paths in $\dyp{m}{n}$.
\end{theorem}
Without delving into details, the idea of the bijection is as follows. 
We identify each dominant region in $\rR_+^m(A_{n-1})$
with its Shi tableau. For each $m$-Dyck path, we ignore its
first step
\footnote{The first step is always 0 and corresponds to
$s_1=0$.} and  we rewrite its step sequence as
$\lambda_{n-1}\leq\cdots\leq\lambda_{1}$ with $0\leq\lambda_i \leq (n-1-i)m$.
Theorem \ref{bij:kij} shows that there  exists a unique way to construct a Shi
tableau of a dominant region in $\rR_+^m(A_{n-1})$,  whose $i$-th row has
coordinates which sum up to $\lambda_i$, for all $1\leq{i}\leq{n-1}$
(see Figure \ref{fig:shifted_partitions_to_tableau}).

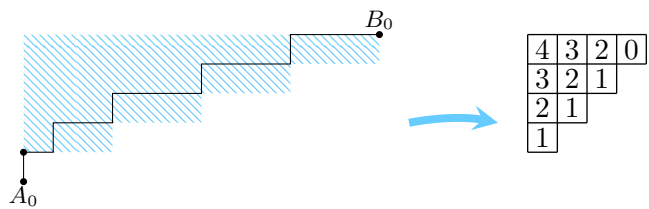
\begin{figure}[h]
\begin{center}
\begin{tikzpicture}[scale=0.65]

\begin{scope}[scale=0.6,xshift=0cm]

\fill[pattern=north west lines, pattern 
color=mycyan](0,0)--(3,0)--(3,1)--(6,1)--(6,2)--(9,2)--(9,3)--(12,3)--(12,4)--
(0,4);

\filldraw (0,-1) circle (0.1cm);
\filldraw (0,0) circle (0.1cm);
\filldraw (12,4) circle (0.1cm);
\draw 
(0,-1)--(0,0)--(1,0)--(1,1)--(3,1)--(3,2)--(6,2)--(6,3)--(9,3)--(9,4)--(12,4);
\node at (0,-1.5){\fs $A_0$};
\node at (12,4.5){\fs $B_0$};

  \draw[xshift=13cm,yshift=1cm,mycyan,line width=3pt, 
  -stealth,rotate=0] 
  (0,0).. controls (1,0.2) and (2,0.2)  .. (3,0);
\end{scope}

\begin{scope}[scale=0.6,xshift=17cm]
\draw[line width= 0.5pt](0,0)--(1,0);
\draw[line width= 0.5pt](1,0)--(1,4);
\draw[line width= 0.5pt](0,0)--(0,4);
\draw[line width= 0.5pt](2,1)--(2,4);  
\draw[line width=0.5pt](3,2)--(3,4);
\draw[line width= 0.5pt](4,3)--(4,4);
\draw[line width= 0.5pt](0,1)--(2,1);
\draw[line width= 0.5pt] (0,2)--(3,2);
\draw[line width= 0.5pt] (0,3)--(4,3);
\draw[line width= 0.5pt] (0,4)--(4,4);
\node at(.5,3.5){$4$};\node at(1.5,3.5){3};\node at(2.5,3.5){2};\node at(3.5, 
3.5){0};
\node at(.5,2.5){3};\node at(1.5,2.5){2};\node at(2.5,2.5){1};
\node at(.5,1.5){2};\node at(1.5,1.5){1};
\node at(.5,.5){1}; 
\end{scope}
\end{tikzpicture}
\end{center}

\caption{For $m=3$, the bijection of Theorem \ref{bij:kij} maps
  the Dyck path $(1,3,6,9)$ to the region Shi tableau depicted on the 
  right.    }
\label{fig:shifted_partitions_to_tableau}

\end{figure}

%
%\begin{figure}[h]
%  \centering
%  \includegraphics[width=0.5\linewidth]{fig:shifted_partitions_to_tableau}
%  \caption{For $m=3$, the bijection of Theorem \ref{bij:kij} maps
%    the Dyck path $(1,3,6,9)$ to the region Shi tableau depicted on the 
%    right.    }
%  \label{fig:shifted_partitions_to_tableau}
%\end{figure}

\section*{Acknowledgments.}
The authors are grateful to Eli Bagno for helpful discussions and to Philippe
Nadeau for bringing \cite{mo-lpc-79} to their attention.
The first author was supported by a ``Back-to-Research Grant'' of the University of Vienna.

\bibliographystyle{authordate1}

\bibliography{projections}

\appendix
\renewcommand{\thesection}{\Alph{section}}
\renewcommand{\thetheorem}{\Alph{section}.\arabic{theorem}}

\section{Shi tableaux for dominant regions}
\label{appA}
Throughout the paper, we use the fact that each dominant region in \mshi{\Phi} 
is represented by its $m$-minimal alcove, and build all our arguments upon 
this. In this section we show how we encode each dominant region $\rR$ in a  
tableau $T_{\rR}$, and how we retrieve  the Shi tableau of its $m$-minimal 
alcove from $T_{\rR}$. 

The \emph{Shi tableau of a  dominant region}   $\rR$ in \mshi{\Phi} is the set 
$T_{\rR}= \{r_{\alpha}$, $\alpha\in\Phi^+\}\subset\mathbb{N}$ of coordinates,
where $r_{\alpha}$ counts the number of integer translates $H_{\alpha,k}$
of $H_{\alpha,0}$, that separate $\rR$ from the origin. Since, by definition of 
\mshi{\Phi}, the integer $k$ is bounded by $m$, we deduce that  
$0\leq{}r_{\alpha}\leq{}m.$ As before, the coordinates $r_{\alpha}$ are 
arranged in a tableau according to the root system. The coordinates of 
$T_{\rR}$ yield the face defining inequalities for the region $\rR$. More 
precisely, for each $\alpha\in\Phi^+$ it is $r_{\alpha}=r$ if and only if, for 
all $x\in\rR$:
\begin{equation}
\label{equ:ineq_alcove}
\begin{aligned}
r  & <\langle{}x,\alpha\rangle<r+1 && \mbox{ if }r<m, \\
m & <\langle{}x,\alpha\rangle  && \mbox{ if }r=m.
\end{aligned}
\end{equation}
The above inequalities imply the the following \emph{Shi conditions on a
  region}: for $\alpha,\beta,\gamma\in\Phi^+$ with
$\alpha+\beta=\gamma$, it holds:
\begin{equation}
r_{\alpha}=\begin{cases}
r_{\beta}+r_{\gamma}+\delta_{\beta,\gamma}\hspace{0.1cm}
\mbox{ if }\;\;r_{\beta}+r_{\gamma}<m,\\
m\hspace{2cm}\mbox{otherwise},
\end{cases}
\end{equation}
where $\delta_{\beta,\gamma}\in\{0,1\}.$
\medskip

As we mentioned in  Section~\ref{sec:mShi}, each dominant
region in \mshi{\Phi} is uniquely represented by its $m$-minimal alcove. One
can switch from the tableau of the region $\rR$ to that of its $m$-minimal
alcove $\alc_{\rR}$ and vice versa, as indicated by the following lemma (see
also Figure~\ref{fig:ex_tableau}).

\begin{lemma}\cite[Section 3]{ath-rgcn-05}
  \label{lem:region_to_alcove}
  Let $\rR$ be a  dominant region in \mshi{\Phi} with $m$-minimal alcove
  $\alc_{\rR}$. Let also $T_{\rR}=\{r_{\alpha}:\alpha\in\Phi^+\}$ and
  $T=\{k_{\alpha}:\alpha\in\Phi^+\}$ be the Shi tableau of $\rR$ and
  $\alc_{\rR}$ respectively. Then, for each $\alpha\in\Phi^+$ the
  following relations hold:
  \begin{enumerate}[(i)]
    \item $r_{\alpha}=\min\{m,k_\alpha\}$\quad and
    \item $k_{\alpha}=\max\{k_{\beta}+k_{\gamma}: \alpha=\beta+\gamma
    \mbox{ with } \beta,\gamma\in\Phi^+\}$.
  \end{enumerate}
\end{lemma}

The following corollary, which is a direct consequence of Lemma
\ref{lem:region_to_alcove}, is used in the proof of Theorem 
\ref{thm:mconditions2}. 
We could alternatively have used arguments based on the 
fact that a region as well as its $m$-minimal alcove have the same 
set of separating walls. We, however, add this appendix  
in order to be able to present explicit examples associating 
regions to their $m$-minimal alcove (see Figure 
\ref{fig:step_by_step} and Appendix \ref{appB}). 
\begin{corollary}
  \label{cor:same_region}
  The dominant alcove $\alc$  with Shi tableau
  $T=\{k_{ij}:1\leq{}i\leq{}j\leq{}n-1\}$ lies in the 	dominant region
  $\rR$  with Shi tableau
  $T_{\rR}=\{r_{ij}=\min\{k_{ij},m\}: \;1\leq{}i\leq{}j\leq{}n-1 \}$.
\end{corollary}

\section{Another example of \bj}
\label{appB}

In this appendix we present one more example of the map  \bj. More precisely, 
ignoring the first and last step of \bj\ (that biject  each region 
to its $m$-minimal alcove), we consider an $m$-minimal alcove $w\aA_0$ in 
$\aA_+(A_{n-1})$, we compute its  $m$-admissible set $\mgk{m}{w}$
and we evaluate the Shi tableaux  $T_i$ of the $m$-minimal alcoves in 
$\aA_+(C_{n})$ corresponding to $\w{k_i}$, for all $k_i\in\mgk{m}{w}$. Our goal 
is to illustrate how  a copy of the Shi tableau $T$ resides in 
each of the $mn+1$ tableaux $T_i$. 
\smallskip

Let us consider  the  2-minimal alcove $w\aA_0$  $\in\aA_+(A_3)$ 
with $w=[3^{-1},1^0,2^0,4^1]$, whose  Shi tableau $T$ is  \medskip
\begin{center}
\begin{tikzpicture}[scale=0.4]
  \draw[line width= 0.5pt](1,1)rectangle(2,4);  
  \draw[line width=0.5pt](2,2)rectangle(3,4);
  \draw[line width= 0.5pt](3,3)rectangle(4,4);
  \draw[line width= 0.5pt] (1,2)--(2,2);
  \draw[line width= 0.5pt] (1,3)--(3,3);
  \node at(1.5,3.5) {$2$};\node at(2.5,3.5) {$0$};\node at(3.5, 3.5){$0$};
  \node at(1.5,2.5) {$1$};\node at(2.5,2.5) {$0$};
  \node at(1.5,1.5) {$1$};
  \node at(0,3){$T=$}; 
\end{tikzpicture}.
\end{center}
Using Proposition \ref{prop:conditions}, we compute the $2$-admissible set 
$\mgk{w}{2}=\{-8,-4,-2,-1,0,2,3,4,6\}$ of $w$. For each $k_i\in\mgk{w}{2}$ we 
apply the antisymmetric expansion of the $k_i$-shift of $w$ (i.e., the map of 
Proposition \ref{prop:proj alc}), and we obtain  the tableaux listed in Figure 
\ref{big_fig}.  As illustrated in blue, a copy of the tableau $T$ as well as 
one of  its conjugate $T'$, occupy certain rows and columns of each $T_i$.  
Moreover, the values of $k\in\mgk{w}{2}$ have an interesting feature:  as the 
$k_i$'s grow, the tableau $T$ shifts gradually from the top rightmost entries 
of $T_i$ to the bottom ones. So, for example, for the smallest value $k_1$, the 
tableau $T$ occupies the top rightmost entries of $T_1$ (and $T'$ occupies the 
bottom ones), while for the greater value $k_9$, the tableau $T$ occupies the 
bottom entries of $T_9$ (and $T'$ the top rightmost ones). 
{
\begin{figure}[h]  
\begin{tikzpicture}[scale=0.4]
%%%%%%%%%%%%%%%%%%%%%%%%%k=-8%%%%%%%%%%%%%%%%%%%%%%%%%%%%%%%%%%%%%%%%%%	
  \begin{scope}[xshift=-30 cm,yshift=0cm]
    \foreach \j in {0,...,6} {\foreach \i in {0,...,\j} { \draw (\i,\j) 
        rectangle (1+\i,\j+1); } }
%%%%%%    
    \node at(.5, 6.5){6};\node at(1.5,6.5){5};\node at(2.5,6.5){5};
    \node at(3.5,6.5){4};\node at(4.5,6.5){\mk{2}};\node at(5.5,6.5){\mk{0}};
    \node at(6.5,6.5){\mk{0}};
%%%%%    
    \node at(.5, 5.5){5};\node at(1.5, 5.5){4};\node at(2.5,5.5){4}; 
    \node at(3.5,5.5){3};\node at(4.5,5.5){\mk{1}};\node at(5.5,5.5){\mk{0}}; 
%%%%%    
    \node at(.5,4.5){5};\node at(1.5,4.5){4};\node at(2.5,4.5){4}; 
    \node at(3.5,4.5){3};\node at(4.5,4.5){\mk{1}}; 
%%%%%    
    \node at(.5,3.5){4};\node at(1.5,3.5){3};\node at(2.5,3.5){3};
    \node at(3.5,3.5){2}; 
%%%%%    
    \node at(.5,2.5){\mk{2}};\node at(1.5,2.5){\mk{1}};\node 
    at(2.5,2.5){\mk{1}};
%%%%%     
    \node at(.5,1.5){\mk{0}};\node at(1.5,1.5){\mk{0}}; 
%%%%%    
    \node at(.5,.5){\mk{0}}; 
%%%    
     \node at(-1, 3.5){$T_1=$};
    \node[text width=6cm] at (5.5,-2.5) {\footnotesize
      $k_1=-8 = 0^{-2}$ 
      
      $[3^{-1},1^{0},2^0,4^1]\overset{k_1=-8}{\Rightarrow}
      [3^{-3},{1^{-2}},{2^{-2}},{4^{-1}}]$
      \bigskip
      
      $[\mk{3^{-3}},\mk{1^{-2}},\mk{2^{-2}},\mk{4^{-1}},5^1,7^2,8^2,6^3]$};
  \end{scope}
%%%%%%%%%%%%%%%%%%%%%%%%%%%%%%%%%%%%%%%%%%%%%%%%%%%%%%%%%%%%%%%%%%%%%%%
%%%%%%%%%%%%%%%%%%%%%%%%%%k=-4%%%%%%%%%%%%%%%%%%%%%%%%%%%%%%%%%%%%%%%%%	
\begin{scope}[xshift=-17 cm,yshift=0cm]
\foreach \j in {0,...,6} {\foreach \i in {0,...,\j} { \draw (\i,\j) 
    rectangle (1+\i,\j+1); } }
%%%%%%    
\node at(.5, 6.5){4};\node at(1.5,6.5){3};\node at(2.5,6.5){3};
\node at(3.5,6.5){2};\node at(4.5,6.5){\mk{2}};\node at(5.5,6.5){\mk{0}};
\node at(6.5,6.5){\mk{0}};
%%%%%    
\node at(.5, 5.5){3};\node at(1.5, 5.5){2};\node at(2.5,5.5){2}; 
\node at(3.5,5.5){1};\node at(4.5,5.5){\mk{1}};\node at(5.5,5.5){\mk{0}}; 
%%%%%    
\node at(.5,4.5){3};\node at(1.5,4.5){2};\node at(2.5,4.5){2}; 
\node at(3.5,4.5){1};\node at(4.5,4.5){\mk{1}}; 
%%%%%    
\node at(.5,3.5){2};\node at(1.5,3.5){1};\node at(2.5,3.5){1};
\node at(3.5,3.5){0}; 
%%%%%    
\node at(.5,2.5){\mk{2}};\node at(1.5,2.5){\mk{1}};\node 
at(2.5,2.5){\mk{1}};
%%%%%     
\node at(.5,1.5){\mk{0}};\node at(1.5,1.5){\mk{0}}; 
%%%%%    
\node at(.5,.5){\mk{0}}; 
%%%    
     \node at(-1, 3.5){$T_2=$};
\node[text width=5cm] at (5.5,-2.5) {\footnotesize
  $k_2=-4=0^{-1}$
  
  $[3^{-1},1^{0},2^0,4^1]\overset{k_2=-4}{\Rightarrow}
  [{3^{-2}},{1^{-1}},{2^{-1}},{4^{0}}]$
  \bigskip
  
  $[\mk{3^{-2}},\mk{1^{-1}},\mk{2^{-1}},\mk{4^{0}},5^0,7^1,8^1,6^2]$};
\end{scope}  
%%%%%%%%%%%%%%%%%%%%%%%%%%%%k=-2 %%%%%%%%%%%%%%%%%%%%%%%%%%%%%%%%%%%%%%%%
%%%%%%%%%%%%%%%%%%%%%%%%%%%%%%%%%%%%%%%%%%%%%%%%%%%%%%%%%%%%%%%%%%%%%%%	
\begin{scope}[xshift=-4 cm,yshift=0cm]
\foreach \j in {0,...,6} {\foreach \i in {0,...,\j} { \draw (\i,\j) 
    rectangle (1+\i,\j+1); } }
%%%%%%    
\node at(.5, 6.5){2};\node at(1.5,6.5){2};\node at(2.5,6.5){2};
\node at(3.5,6.5){\mk{2}};\node at(4.5,6.5){0};\node at(5.5,6.5){\mk{0}};
\node at(6.5,6.5){\mk{0}};
%%%%%    
\node at(.5, 5.5){2};\node at(1.5, 5.5){2};\node at(2.5,5.5){2}; 
\node at(3.5,5.5){\mk{1}};\node at(4.5,5.5){0};\node at(5.5,5.5){\mk{0}}; 
%%%%%    
\node at(.5,4.5){2};\node at(1.5,4.5){2};\node at(2.5,4.5){2}; 
\node at(3.5,4.5){\mk{1}};\node at(4.5,4.5){0}; 
%%%%%    
\node at(.5,3.5){\mk{2}};\node at(1.5,3.5){\mk{1}};\node at(2.5,3.5){\mk{1}};
\node at(3.5,3.5){1}; 
%%%%%    
\node at(.5,2.5){0};\node at(1.5,2.5){0};\node at(2.5,2.5){0};
%%%%%     
\node at(.5,1.5){\mk{0}};\node at(1.5,1.5){\mk{0}}; 
%%%%%    
\node at(.5,.5){\mk{0}}; 
%%%    
     \node at(-1, 3.5){$T_3=$};
\node[text width=6cm] at (5.5,-2.5) {\footnotesize
  $k_3=-2=2^{-1}$
  
  $[3^{-1},1^{0},2^0,4^1]\overset{k_3=-2}{\Rightarrow}
  [1^{-1},3^{-1},4^{-1},2^1]$
  \bigskip
  
  $[\mk{1^{-1}},\mk{3^{-1}},\mk{4^{-1}},7^{-1},\mk{2^{1}},5^1,6^1,8^1]$};
\end{scope}
\end{tikzpicture}
\vspace{0.5cm}

%%%%%%%%%%%%%%%%%%%%%%%%%%%%%%%%%%%%%%%%%%%%%%%%%%%%%%%%%%%%%%%%%%%%%%%%%%%%%%
%%%%%%%%%%%%%%%%%%%%%%NEXT THREE%%%%%%%%%%%%%%%%%%%%%%%%%%%%%%%%%%%%%%%%%%%%%%

\begin{tikzpicture}[scale=0.4]
%%%%%%%%%%%%%%%%%%%%%%%%%k=-1%%%%%%%%%%%%%%%%%%%%%%%%%%%%%%%%%%%%%%%%%%	
\begin{scope}[xshift=-30 cm,yshift=0cm]
\foreach \j in {0,...,6} {\foreach \i in {0,...,\j} { \draw (\i,\j) 
    rectangle (1+\i,\j+1); } }
%%%%%%    
\node at(.5, 6.5){2};\node at(1.5,6.5){2};\node at(2.5,6.5){\mk{2}};
\node at(3.5,6.5){1};\node at(4.5,6.5){\mk{0}};\node at(5.5,6.5){\mk{0}};
\node at(6.5,6.5){0};
%%%%%    
\node at(.5, 5.5){2};\node at(1.5, 5.5){2};\node at(2.5,5.5){\mk{1}}; 
\node at(3.5,5.5){1};\node at(4.5,5.5){\mk{0}};\node at(5.5,5.5){0}; 
%%%%%    
\node at(.5,4.5){\mk{2}};\node at(1.5,4.5){\mk{1}};\node at(2.5,4.5){\mk{1}}; 
\node at(3.5,4.5){1};\node at(4.5,4.5){0}; 
%%%%%    
\node at(.5,3.5){1};\node at(1.5,3.5){1};\node at(2.5,3.5){1};
\node at(3.5,3.5){0}; 
%%%%%    
\node at(.5,2.5){\mk{0}};\node at(1.5,2.5){\mk{0}};\node 
at(2.5,2.5){0};
%%%%%     
\node at(.5,1.5){\mk{0}};\node at(1.5,1.5){0}; 
%%%%%    
\node at(.5,.5){0}; 
%%%    
     \node at(-1, 3.5){$T_4=$};
\node[text width=6cm] at (5.5,-2.5) {\footnotesize
  $k_4=-1=3^{-1}$
  
  $[3^{-1},1^{0},2^0,4^1]\overset{k_4=-1}{\Rightarrow}
  [2^{-1},4^{-1},1^0,3^1]$
  \bigskip
  
  $[\mk{2^{-1}},\mk{4^{-1}},6^{-1},\mk{1^0},8^0,3^1,5^1,7^1]$};
\end{scope}
%%%%%%%%%%%%%%%%%%%%%%%%%%%%%%%%%%%%%%%%%%%%%%%%%%%%%%%%%%%%%%%%%%%%%%%
%%%%%%%%%%%%%%%%%%%%%%%%%%k=0%%%%%%%%%%%%%%%%%%%%%%%%%%%%%%%%%%%%%%%%%	
\begin{scope}[xshift=-17 cm,yshift=0cm]
\foreach \j in {0,...,6} {\foreach \i in {0,...,\j} { \draw (\i,\j) 
    rectangle (1+\i,\j+1); } }
%%%%%%    
\node at(.5, 6.5){2};\node at(1.5,6.5){\mk{2}};\node at(2.5,6.5){1};
\node at(3.5,6.5){1};\node at(4.5,6.5){\mk{0}};\node at(5.5,6.5){\mk{0}};
\node at(6.5,6.5){0};
%%%%%    
\node at(.5, 5.5){\mk{2}};\node at(1.5, 5.5){\mk{1}};\node at(2.5,5.5){\mk{1}}; 
\node at(3.5,5.5){1};\node at(4.5,5.5){\mk{0}};\node at(5.5,5.5){0}; 
%%%%%    
\node at(.5,4.5){1};\node at(1.5,4.5){\mk{1}};\node at(2.5,4.5){0}; 
\node at(3.5,4.5){0};\node at(4.5,4.5){0}; 
%%%%%    
\node at(.5,3.5){1};\node at(1.5,3.5){1};\node at(2.5,3.5){0};
\node at(3.5,3.5){0}; 
%%%%%    
\node at(.5,2.5){\mk{0}};\node at(1.5,2.5){\mk{0}};\node 
at(2.5,2.5){0};
%%%%%     
\node at(.5,1.5){\mk{0}};\node at(1.5,1.5){0}; 
%%%%%    
\node at(.5,.5){0}; 
%%%    
     \node at(-1, 3.5){$T_5=$};
\node[text width=5cm] at (5.5,-2.5) {\footnotesize
  $k_5=0$
  
  $[3^{-1},1^{0},2^0,4^1]\overset{k_5=0}{\Rightarrow}
  [3^{-1},1^{0},2^0,4^1]$
  \bigskip
  
  $[\mk{3^{-1}},5^{-1},\mk{1^{0}},\mk{2^{0}},7^0,8^0,\mk{4^1},6^1]$};
\end{scope}  
%%%%%%%%%%%%%%%%%%%%%%%%%%%%k=2 %%%%%%%%%%%%%%%%%%%%%%%%%%%%%%%%%%%%%%%%
%%%%%%%%%%%%%%%%%%%%%%%%%%%%%%%%%%%%%%%%%%%%%%%%%%%%%%%%%%%%%%%%%%%%%%%	
\begin{scope}[xshift=-4 cm,yshift=0cm]
\foreach \j in {0,...,6} {\foreach \i in {0,...,\j} { \draw (\i,\j) 
    rectangle (1+\i,\j+1); } }
%%%%%%    
\node at(.5, 6.5){3};\node at(1.5,6.5){\mk{2}};\node at(2.5,6.5){\mk{1}};
\node at(3.5,6.5){\mk{1}};\node at(4.5,6.5){1};\node at(5.5,6.5){1};
\node at(6.5,6.5){1};
%%%%%    
\node at(.5, 5.5){\mk{2}};\node at(1.5, 5.5){0};\node at(2.5,5.5){0}; 
\node at(3.5,5.5){0};\node at(4.5,5.5){\mk{0}};\node at(5.5,5.5){\mk{0}}; 
%%%%%    
\node at(.5,4.5){\mk{1}};\node at(1.5,4.5){0};\node at(2.5,4.5){0}; 
\node at(3.5,4.5){0};\node at(4.5,4.5){\mk{0}}; 
%%%%%    
\node at(.5,3.5){\mk{1}};\node at(1.5,3.5){0};\node at(2.5,3.5){0};
\node at(3.5,3.5){0}; 
%%%%%    
\node at(.5,2.5){1};\node at(1.5,2.5){\mk{0}};\node at(2.5,2.5){\mk{0}};
%%%%%     
\node at(.5,1.5){1};\node at(1.5,1.5){\mk{0}}; 
%%%%%    
\node at(.5,.5){1}; 
%%%    
     \node at(-1, 3.5){$T_6=$};
\node[text width=6cm] at (5.5,-2.5) {\footnotesize
  $k_6=2=2^0$
  
  $[3^{-1},1^{0},2^0,4^1]\overset{k_6=2}{\Rightarrow}
  [1^0,3^0,4^0,2^2]$
  \bigskip
  
  $[6^{-2},\mk{1^0},\mk{3^0},\mk{4^0},5^0,7^0,8^0,\mk{2^2}]$};
\end{scope}
\end{tikzpicture}
\vspace{0.5cm}

%%%%%%%%%%%%%%%%%%%%%%%%%%%%%%%%%%%%%%%%%%%%%%%%%%%%%%%%%%%%%%%%%%%%%%%%%%%%%%
%%%%%%%%%%%%%%%%%%%%%%LAST THREE%%%%%%%%%%%%%%%%%%%%%%%%%%%%%%%%%%%%%%%%%%%%%%

\begin{tikzpicture}[scale=0.4]
%%%%%%%%%%%%%%%%%%%%%%%%%k=3%%%%%%%%%%%%%%%%%%%%%%%%%%%%%%%%%%%%%%%%%%	
\begin{scope}[xshift=-30 cm,yshift=0cm]
\foreach \j in {0,...,6} {\foreach \i in {0,...,\j} { \draw (\i,\j) 
    rectangle (1+\i,\j+1); } }
%%%%%%    
\node at(.5, 6.5){3};\node at(1.5,6.5){2};\node at(2.5,6.5){\mk{2}};
\node at(3.5,6.5){\mk{1}};\node at(4.5,6.5){1};\node at(5.5,6.5){1};
\node at(6.5,6.5){\mk{1}};
%%%%%    
\node at(.5, 5.5){2};\node at(1.5, 5.5){1};\node at(2.5,5.5){\mk{0}}; 
\node at(3.5,5.5){\mk{0}};\node at(4.5,5.5){0};\node at(5.5,5.5){0}; 
%%%%%    
\node at(.5,4.5){\mk{2}};\node at(1.5,4.5){\mk{0}};\node at(2.5,4.5){0}; 
\node at(3.5,4.5){0};\node at(4.5,4.5){\mk{0}}; 
%%%%%    
\node at(.5,3.5){\mk{1}};\node at(1.5,3.5){\mk{0}};\node at(2.5,3.5){0};
\node at(3.5,3.5){0}; 
%%%%%    
\node at(.5,2.5){1};\node at(1.5,2.5){0};\node at(2.5,2.5){\mk{0}};
%%%%%     
\node at(.5,1.5){1};\node at(1.5,1.5){0}; 
%%%%%    
\node at(.5,.5){\mk{1}}; 
%%%    
     \node at(-1, 3.5){$T_7=$};
\node[text width=6cm] at (5.5,-2.5) {\footnotesize
  $k_7=3=3^0$
  
  $[3^{-1},1^{0},2^0,4^1]\overset{k_7=3}{\Rightarrow}
  [2^0,4^0,1^1,3^2]$
  \bigskip
  
  $[7^{-2},5^{-1},\mk{2^0},\mk{4^0},6^0,8^0,\mk{1^1},\mk{3^2}]$};
\end{scope}
%%%%%%%%%%%%%%%%%%%%%%%%%%%%%%%%%%%%%%%%%%%%%%%%%%%%%%%%%%%%%%%%%%%%%%%
%%%%%%%%%%%%%%%%%%%%%%%%%%k=4%%%%%%%%%%%%%%%%%%%%%%%%%%%%%%%%%%%%%%%%%	
\begin{scope}[xshift=-17 cm,yshift=0cm]
\foreach \j in {0,...,6} {\foreach \i in {0,...,\j} { \draw (\i,\j) 
    rectangle (1+\i,\j+1); } }
%%%%%%    
\node at(.5, 6.5){3};\node at(1.5,6.5){2};\node at(2.5,6.5){2};
\node at(3.5,6.5){\mk{2}};\node at(4.5,6.5){1};\node at(5.5,6.5){\mk{1}};
\node at(6.5,6.5){\mk{1}};
%%%%%    
\node at(.5, 5.5){2};\node at(1.5, 5.5){1};\node at(2.5,5.5){1}; 
\node at(3.5,5.5){\mk{0}};\node at(4.5,5.5){0};\node at(5.5,5.5){\mk{0}}; 
%%%%%    
\node at(.5,4.5){2};\node at(1.5,4.5){1};\node at(2.5,4.5){1}; 
\node at(3.5,4.5){\mk{0}};\node at(4.5,4.5){0}; 
%%%%%    
\node at(.5,3.5){\mk{2}};\node at(1.5,3.5){\mk{0}};\node at(2.5,3.5){\mk{0}};
\node at(3.5,3.5){0}; 
%%%%%    
\node at(.5,2.5){1};\node at(1.5,2.5){0};\node at(2.5,2.5){0};
%%%%%     
\node at(.5,1.5){\mk{1}};\node at(1.5,1.5){\mk{0}}; 
%%%%%    
\node at(.5,.5){\mk{1}}; 
%%%    
     \node at(-1, 3.5){$T_8=$};
\node[text width=5cm] at (5.5,-2.5) {\footnotesize
  $k_8=4=0^1$
  
  $[3^{-1},1^{0},2^0,4^1]\overset{k_8=4}{\Rightarrow}
  [3^0,1^1,2^1,4^2]$
  \bigskip
  
  $[8^{-2},5^{-1},6^{-1},\mk{3^0},7^0,\mk{1^1},\mk{2^1},\mk{4^2}]$};
\end{scope}  
%%%%%%%%%%%%%%%%%%%%%%%%%%%%k=6 %%%%%%%%%%%%%%%%%%%%%%%%%%%%%%%%%%%%%%%%
%%%%%%%%%%%%%%%%%%%%%%%%%%%%%%%%%%%%%%%%%%%%%%%%%%%%%%%%%%%%%%%%%%%%%%%	
\begin{scope}[xshift=-4 cm,yshift=0cm]
\foreach \j in {0,...,6} {\foreach \i in {0,...,\j} { \draw (\i,\j) 
    rectangle (1+\i,\j+1); } }
%%%%%%    
\node at(.5, 6.5){5};\node at(1.5,6.5){3};\node at(2.5,6.5){3};
\node at(3.5,6.5){3};\node at(4.5,6.5){\mk{2}};\node at(5.5,6.5){\mk{1}};
\node at(6.5,6.5){\mk{1}};
%%%%%    
\node at(.5, 5.5){1};\node at(1.5, 5.5){1};\node at(2.5,5.5){1}; 
\node at(3.5,5.5){1};\node at(4.5,5.5){\mk{0}};\node at(5.5,5.5){\mk{0}}; 
%%%%%    
\node at(.5,4.5){3};\node at(1.5,4.5){1};\node at(2.5,4.5){1}; 
\node at(3.5,4.5){1};\node at(4.5,4.5){\mk{0}}; 
%%%%%    
\node at(.5,3.5){3};\node at(1.5,3.5){1};\node at(2.5,3.5){1};
\node at(3.5,3.5){1}; 
%%%%%    
\node at(.5,2.5){\mk{2}};\node at(1.5,2.5){\mk{0}};\node at(2.5,2.5){\mk{0}};
%%%%%     
\node at(.5,1.5){\mk{1}};\node at(1.5,1.5){\mk{0}}; 
%%%%%    
\node at(.5,.5){\mk{1}}; 
%%%    
     \node at(-1, 3.5){$T_9=$};
\node[text width=6cm] at (5.5,-2.5) {\footnotesize
  $k_9=6=2^1$
  
  $[3^{-1},1^{0},2^0,4^1]\overset{k_9=6}{\Rightarrow}
  [1^{1},3^{1},4^{1},2^3]$
  \bigskip
  
  $[6^{-3},5^{-1},7^{-1},8^{-1},\mk{1^1},\mk{3^1},\mk{4^1},\mk{2^3}]$};
\end{scope}
\end{tikzpicture}
\caption{}
\label{big_fig}
\end{figure}
}
The tableaux computed above correspond to $2$-minimal alcoves of different   
regions in $\rR^2_+(C_4)$. Indeed, in view of Corollary \ref{cor:same_region}, 
if, in each tableau, we replace every entry $k_{ij}$ by $\min\{k_{ij},2\}$, 
we have no coincidences among the resulting tableaux, which verifies that 
the regions they represent are different.

\end{document}